\documentclass[11pt]{amsart}

\usepackage[top=1in, bottom=1in, left=1.3in, right=1.3in]{geometry}
\usepackage{changes}

\usepackage{amsmath}
\usepackage{amssymb}
\usepackage{amsthm}
\usepackage{verbatim}
\usepackage{graphicx}
\usepackage{graphics}
\usepackage{color}
\usepackage{enumerate}
\usepackage{mathrsfs}
\usepackage{booktabs}
\usepackage[toc,page]{appendix}
\usepackage{enumerate}
\usepackage{mathrsfs}
\usepackage{fancyhdr}
\usepackage{amsopn, amstext, amscd,pifont}
\usepackage{amssymb,bm, cite}
\usepackage[all]{xy}
\usepackage{color}
\usepackage{graphicx,hyperref,float}
\usepackage{float}
\usepackage{listings}
\usepackage{setspace}
\usepackage{url}
\usepackage{paralist}
\usepackage{subcaption}
\usepackage{stmaryrd}
\usepackage{tikz-cd}
\usepackage{blkarray, bigstrut}
\usepackage{xparse}

\usepackage{graphicx,pgf}
\usepackage{float}
\usepackage{enumitem}

\usepackage{fancyhdr}

\numberwithin{equation}{section}

\theoremstyle{plain}

\newtheorem{theorem}{Theorem}[section]
\newtheorem{proposition}[theorem]{Proposition}
\newtheorem{lemma}[theorem]{Lemma}
\newtheorem{corollary}[theorem]{Corollary}

\theoremstyle{definition}
\newtheorem{definition}[theorem]{Definition}
\newtheorem{example}[theorem]{Example}
\newtheorem{remark}[theorem]{Remark}

\graphicspath{{images/}}

\renewcommand{\a}{\alpha}

\renewcommand{\d}{\delta}

\newcommand{\g}{\gamma}

\renewcommand{\l}{\lambda}

\renewcommand{\r}{\rho}
\newcommand{\s}{\sigma}

\renewcommand{\th}{\theta}

\newcommand{\gF}{{\mathfrak{F}}}

\newcommand{\Acal}{{\mathcal A}}

\newcommand{\Ccal}{{\mathcal C}}

\newcommand{\Jcal}{{\mathcal J}}

\newcommand{\Ocal}{{\mathcal O}}

\newcommand{\Scal}{{\mathcal S}}
\newcommand{\Tcal}{{\mathcal T}}
\newcommand{\Ucal}{{\mathcal U}}

\newcommand{\ZZ}{{\mathbb{Z}}}

\newcommand{\PP}{{\mathbb{P}}}

\newcommand{\NN}{{\mathbb{N}}}

\newcommand{\Sing}{\operatorname{\bf Sing}}

\newcommand{\GO}{\operatorname{GO}}

\newcommand{\trace}{\operatorname{trace}}

\newcommand{\diam}{\operatorname{diam}}

\newcommand{\Wbar}{\overline{W}}

\newcommand{\Kbar}{\overline{K}}

\newcommand{\Dbar}{\overline{D}}

  {\end{list}}

\title[Zeta function and entropy]{Zeta function and entropy for non-archimedean subhyperbolic dynamics}
\author{Liang-Chung Hsia}
\address{Department of Mathematics, National Taiwan Normal University, Taipei, Taiwan, ROC}
\email{hsia@math.ntnu.edu.tw}

\author{Hongming Nie}
\address{Los Angeles, CA 90089, USA}
\email{hongming.i.nie@gmail.com}

\author{Chenxi Wu}
\address{Department of Mathematics, University of Wisconsin-Madison, 480 Lincoln Drive, Madison, WI 53706, USA}
\email{cwu367@math.wisc.edu}

\begin{document}

\begin{abstract}

Let $K$ be a complete non-archimedean field of characteristic $0$ 
equipped with a discrete valuation. We establish the rationality of  the Artin-Mazur zeta function on the Julia set for any subhyperbolic rational map defined over $K$ with a compact Julia set. Furthermore, we conclude that the topological entropy on the Julia set of such a map is given by the logarithm of a weak Perron number. Conversely, we construct a (sub)hyperbolic rational map defined over $K$ with  compact Julia set whose  topological entropy on the Julia set equals  the  logarithm of a given weak Perron number. This extends Thurston's work on the entropy  for postcritically finite interval self-maps 
to the non-archimedean setting.
\end{abstract}
\maketitle

\tableofcontents

\section{Introduction}\label{sec:intro}

The  Artin-Mazur zeta function and topological entropy are deeply intertwined in the study of dynamical systems. This paper aims to investigate these two fundamental objects within the framework of non-archimedean dynamics. Specifically, we  focus on systems generated by the iterative action of certain rational maps defined over non-archimedean fields.  

The  Artin-Mazur zeta function, introduced by Artin and Mazur \cite{Artin65}, encodes information about the cardinality of periodic points. Its radius of convergence provides a meaningful reflection of the  topological entropy for the system. In the setting of differential dynamics, the (ir)rationality of the Artin-Mazur zeta function has been extensively studied. Smale conjectured, and Manning later proved, that the Artin-Mazur zeta function is rational when the dynamics satisfies the hyperbolicity condition known as Axiom A \cite{Manning71}.  In the context of  complex dynamics on the Riemann sphere, Hinkkanen \cite{Hinkkanen94} established the rationality of the Artin-Mazur zeta function for rational maps.  In real dynamics, Milnor and Thurston \cite{Milnor88} studied the Artin-Mazur zeta function for interval maps, see also \cite{Olivares24}. Meanwhile, the topological entropy has been well understood for complex rational maps \cite{Gromov03} and thoroughly examined in the context of real dynamics \cite{Milnor88}. Notably, Thurston \cite{Thurston22} provided a complete description of the entropy for postcritically finite interval maps.

Non-archimedean dynamics arises from various contexts, including number theory, degenerations of complex or real dynamical systems, physics, information theory, and even cognitive science and psychology. In the setting of non-archimedean dynamics, when the ground field is algebraically closed, the study of the Artin-Mazur zeta function and the topological entropy  has already appeared in \cite{Bajpa17, Bridy16, Favre10, Favre22, Lee15}.  In this paper, we focus on univariate rational maps of degree at least $2$ defined over a non-archimedean field.  Typically, in our case the ground  field does not need to be algebraically closed.

Throughout this paper, we fix a field $K$ which is complete with respect to a nontrivial 
non-archimedean absolute value $|\cdot|_K$. Let $\phi(z)\in K(z)$ be a rational map of degree $d\ge 2$. Then $\phi$ induces a self-map on the projective line $\mathbb{P}^{1}(K)$. The {\em $K$-Fatou set} of $\phi$, denoted by $\mathcal{F}(\phi,K)$, is the maximal domain of equicontinuity of $\{\phi^{\circ n}\}_{n\ge 1}$ on $\mathbb{P}^1(K)$. The {\em $K$-Julia set} is then defined by $\mathcal{J}(\phi,K)=\mathbb{P}^1(K)\backslash\mathcal{F}(\phi,K)$. Here $\phi^{\circ n}$ denotes the $n$-times composition of $\phi$. Unlike in complex dynamics,  the Julia set $\mathcal{J}(\phi,K)$ could be empty; for instance, if $\phi$ has good reduction, then $\phi$ has no Julia set for any extension of $K$, see \cite[Theorem 2.17]{Silverman07}. 
Moreover, by definition, $\mathcal{J}(\phi,K)$ is forward invariant under $\phi$, that is, $\phi(\mathcal{J}(\phi,K))\subseteq \mathcal{J}(\phi,K)$,  and hence $\phi$ induces a dynamical system on $\mathcal{J}(\phi,K)$.

Note that any periodic point in $\mathcal{J}(\phi,K)$ must necessarily be a repelling periodic point for $\phi$. In our analysis, we focus on the Artin-Mazur zeta function $\zeta_\phi(K, t)$ associated with the dynamical system $(\mathcal{J}(\phi,K), \phi)$. That is, 
$$\zeta_\phi(K, t):=\exp\left(\sum_{m=1}^\infty\, \#\mathrm{Fix}^*(K, \phi^{\circ m})\cdot\frac{t^m}{m}\right),$$
where $\mathrm{Fix}^*(K, \phi^{\circ m})$ represents the set of repelling fixed points of $\phi^{\circ m}$. It is a formal power series  in $\mathbb{Z}[[t]]$, see \S\ref{sec:AM}.

Our first main result concerns with identifying the conditions under which the Artin-Mazur zeta function 
$\zeta_\phi(K, t)$ is a rational function in $t$. 
To specify these conditions, let us denote the set of  $K$-rational critical points of $\phi$ by $\mathrm{Crit}(\phi,K)$, that is, $\mathrm{Crit}(\phi,K)$ consists of the points $c\in\mathbb{P}^1(K)$ for which $\phi$ is not bijective on any (sufficiently) small neighborhood of $c$.
Inspired by \cite{Benedetto01} and \cite{Fan21}, we classify the map $\phi$ as {\em hyperbolic} if $\mathrm{Crit}(\phi,K)\cap \mathcal{J}(\phi,K) = \emptyset$, and we call $\phi$  \emph{subhyperbolic} if all points in $\mathcal{J}(\phi,K)\cap\mathrm{Crit}(\phi,K)$ are eventually periodic under $\phi$. By definitions, any hyperbolic rational map is subhyperbolic.

\begin{theorem}\label{ratzeta}
Let $K$ be a complete non-archimedean field of characteristic $0$, 
and let $\phi\in K(z)$  be a subhyperbolic rational map of degree at least $2$ with nonempty and compact $\mathcal{J}(\phi,K)$.  
Then  the Artin-Mazur zeta function $\zeta_\phi(K, t)$ is the quotient of two coprime polynomials  in $\mathbb{Z}[t]$ with the numerator a product of cyclotomic polynomials. 
\end{theorem}

If $\mathbb{P}^1(K)$ is compact, then $\mathcal{J}(\phi,K)$ is automatically compact. If $\mathbb{P}^1(K)$ is not compact (for instance, the residue field of $K$ is infinite), then $\mathcal{J}(\phi,K)$ could still be a compact subspace of $\mathbb{P}^{1}(K)$.  We refer to \cite{Hsia96,Kiwi22,Trucco14} for examples of rational maps with compact Julia sets defined over a non-archimedean field with an infinite residue field. 

We deduce Theorem \ref{ratzeta} from a general framework for continuous maps on non-archimedean spaces, as discussed in \S\ref{sec:zeta-ra}. More precisely, we impose  constrains on the singular points of such continuous maps (analogous to critical points of rational maps), and then establish the finiteness of periodic points for a given period, along with the rationality of the Artin-Mazur zeta function. A key  ingredient is to encode the system $(\mathcal{J}(K,\phi),\phi)$ with a symbolic dynamical system $(\Sigma_B,\sigma)$ of finite states, see \S\ref{sec:factoring}. Additionally, we provide a precise formula for the Artin-Mazur zeta function in \S\ref{sec:pr}.  The assumption that $K$ has characteristic zero is  used only in the proof of Lemma \ref{lem:rational-homo}. 

As a consequence of Theorem \ref{ratzeta}, the topological entropy $\mathfrak{h}(\phi,K)$ on the Julia set of a subhyperbolic map $\phi\in K(z)$ of degree at least $2$ with compact $\mathcal{J}(\phi,K)$ is shown to be the logarithm of an algebraic number. Furthermore, we demonstrate that this entropy is, in fact, the logarithm of a weak Perron number. Recall that a real positive algebraic integer $\lambda$ is 
called a \emph{weak Perron number} if the complex absolute values of its Galois conjugates $\lambda'\in\mathbb{R}$ 
satisfy $|\lambda'|\leq \lambda$; 
equivalently, some power of $\lambda$ is a Perron number that is an algebraic integer $\ge 1$ whose remaining Galois conjugates have strictly smaller absolute value (see \cite[Theorem 3]{Lind84} and \cite[Proposition 5.1]{Thurston22}).
Then our result naturally raises the question: Which weak Perron numbers can be realized as the exponential of the topological entropy on the Julia set of a subhyperbolic map over a non-archimedean field?

 Our subsequent result provides a realization of the set of weak Perron numbers as the topological entropy on the Julia sets for subhyperbolic maps  in $K(z)$ with compact Julia sets, provided that $K$ is discretely valued, presenting an analogue to the result \cite[Theorem 1.2]{Thurston22} of Thurston. 

\begin{theorem}\label{main}
Let $K$ be a complete non-archimedean field of characteristic $0$ 
equipped with a discrete valuation. Then the set 
$$\left\{\exp(\mathfrak{h}(\phi,K)): \phi\in K(z)\text{ subhyperbolic with compact } \mathcal{J}(\phi,K)\subseteq \mathbb{P}^1(K)\right\}$$
equals the set of weak Perron numbers.
\end{theorem}


The encoding of the dynamical system $(\mathcal{J}(\phi, K),\phi)$ with $(\Sigma_B,\sigma)$ implies that the entropy 
$\mathfrak{h}(\phi,K)$ can be reformulated as 
$$\mathfrak{h}(\phi,K)=\limsup_{m\rightarrow\infty}\frac{1}{m}\log\#\mathrm{Fix}^*(K, \phi^{\circ m}).$$
The Fatou components in $\mathcal{F}(K,\phi)$ are well-defined and classified in \cite{Benedetto02} though $\mathbb{P}^1(K)$ is totally disconnected.  Since $\phi$ is eventually non-expending on any eventually periodic Fatou component,  such a component carries zero entropy.  Consequently,  if $\mathbb{P}^1(K)$ has no wandering domains, then the entropy $\mathfrak{h}(\phi,K)$ equals the topological entropy $h_{\mathrm{top}}(\phi,K)$ of $\phi$ on $\mathbb{P}^1(K)$. 
In particular, if $\mathbb{P}^1(K)$ contains $2\deg\phi-2$ critical points of $\phi$, counted with multiplicities, then the subhyperbolicity of $\phi$ implies that  $\mathcal{F}(K,\phi)$ has no wandering domain (see, e.g., \cite[Theorem 1.2]{Benedetto00}), and hence $\mathfrak{h}(\phi,K)=h_{\mathrm{top}}(\phi,K)$; we refer to \cite[Chapter 11]{Benedetto19} for a study of wandering domains. Moreover, the Poincar\'e recurrence theorem implies that any invariant probability measure on $\mathbb{P}^1(K)$ cannot charge the wandering domains, and consequently, provided that $\mathbb{P}^1(K)$ is compact, the variational principle yields $\mathfrak{h}(\phi,K)=h_{\mathrm{top}}(\phi,K)$.

 Although our Theorem \ref{main} has a superficial resemblance to Thurston's, our arguments differ from Thurston's in key respects. In Thurston's case, since $\mathbb{R}$ has a linear order, a Markov decomposition can be obtained directly from the forward orbit of the critical points, which induces  that the exponential of any entropy is a weak Perron number. However, non-archimedean fields do not admit such linear orders, which makes the construction of  a desired Markov decomposition  significantly more conceptual and complicated. Once such a Markov decomposition is obtained, as aforementioned, the direction in Theorem \ref{main} that $\exp(\mathfrak{h}(\phi,K))$ is a weak Perron number follows from the symbolic dynamics developed for Theorem \ref{ratzeta}. For the converse direction, given a weak Perron number $\lambda\ge 1$, we construct a natural (0,1)-matrix $A$ with leading eigenvalue $\lambda$ that  exhibits properties reflecting the non-archimedean nature. We then define a piecewise linear map on finitely many disks in  $K$ with $A$ as its  adjacency  matrix. By applying a result from \cite{Nopal23}, we appropriately glue the pieces of this map to obtain a rational map in $K(z)$. After verifying its dynamics, we conclude that this map is hyperbolic and has entropy $\log\lambda$. For more details, we refer to \S\ref{sec:entropy}.

\smallskip
\noindent \textbf{Outline.}
The paper is organized as follows. 
In \S\ref{sec:pre}, we provide  the background material on non-archimedean subhyperbolic maps and the Artin-Mazur zeta function. In \S\ref{sec:zeta-ra}, assuming the compactness of the Julia set, we describe the dynamics on Julia set of a continuous map on a non-archimedean metric space under certain assumptions, named normality and homogeneity, on the singular points. We begin by encoding the Julia dynamics with a symbolic dynamical system of countably many states and then upgrade it to a symbolic dynamical system of finitely many states. In this section, we demonstrate that the Artin-Mazur zeta function for such a continuous map is rational. In \S\ref{sec:rationalmap}, we focus on the case of rational maps. We first prove Theorem \ref{ratzeta} in \S\ref{sec:proof-rational} by verifying that the rational maps satisfy the assumptions in previous section. Then,  in \S\ref{sec:entropy}, we establish Theorem \ref{main}. Finally, in \S\ref{sec:example}, we provide concrete examples for the Artin-Mazur zeta function and
the entropy, using the method developed in this paper.

\smallskip
\noindent\textbf{Acknowledgements.} The work was partially done  during  the second author's visit to the first author in the Department of Mathematics at National Taiwan Normal University (NTNU) and the third author's visit to the second author in the Institute for Mathematical Sciences (IMS) at Stony Brook University. The authors are grateful to the  Department of Mathematics at NTNU and the IMS for their generous hospitality. The third author is supported by Simons Collaboration Grant, Award Number 850685.

\section{Preliminaries}\label{sec:pre}

In this section, we present the necessary background material. In \S\ref{sec:map}, we introduce subhyperbolic continuous maps  on  non-archimedean spaces, extending the concept of subhyperbolic rational maps defined over non-archimedean fields. In \S\ref{sec:AM}, we discuss the Artin-Mazur zeta functions for continuous maps  on  non-archimedean spaces, for finite directed graphs and for subshifts of finite type, respectively.

\subsection{Non-archimedean subhyperbolic maps}\label{sec:map}

Let $X$ be a metric space associated with a metric $\rho$. We say that $X$ is \emph{non-archimedean} if the metric $\rho$ satisfies the strong triangle inequality, that is, $\rho(x, y)\leq \max\{\rho(x, z), \rho(y, z)\}$ for any points $x, y, z\in X$. Given a non-archimedean metric space $(X,\rho)$, a closed disk in $X$ is a set of form 
$$\overline{D}(a,r):=\{x\in X: \rho(x,a)\le r\}$$
 for some $a\in X$ and $r\ge0$; and  an open disk in $X$ is a set of form 
$$D(a,r):=\{x\in X: \rho(x,a)< r\}$$
 for some $a\in X$ and $r>0$. In the following, we call a disk trivial if it is a singleton (i.e. $r=0$). 
 We will use $U$ to denote a disk in $X$ without specifying whether it is open or closed, unless stated otherwise. Under the topology induced by $\rho$, any nontrivial disk in $X$ is a clopen set. Moreover, if two disks intersect, then one contains the other.

 A continuous map $f: X\to X$ induces a non-archimedean dynamical system $(X, f)$. For any $n\ge 1$, write $f^{\circ n}$ for the $n$-times iterate of $f$. For a nonempty subset $Y$ of $X$,  we use the following standard notations for the \emph{forward} and \emph{grand} orbits of $Y$, respectively:  
 \begin{align*}
 \Ocal_{f}(Y) & := \bigcup_{n\ge 0}  f^{\circ n} \left(Y\right),  \\
 \mathrm{GO}(Y)& :=\{x\in X: \exists\ n\ge 0, m\ge 0,\text{and}\  y\in Y\ \text{such that}\ f^{\circ n}(x)=f^{\circ m}(y)\}. 
\end{align*}

A  point $x\in X$ is called a \emph{periodic point of period $n$} if $f^{\circ n}(x)=x$.  In this case, the forward orbit 
$\Ocal_{f}(x) = \{f^{\circ j}(x)\mid j\ge 0\} = \{x,f(x),\dots, f^{\circ (n-1)}(x)\}$ of $x$ is called a \emph{periodic cycle} of length $n$. Moreover, if for any small disk $U$ containing $x$, we have  $U\subsetneq f^{\circ n}(U)$, then $x$ is said to be a \emph{repelling} periodic point and its forward orbit is referred to as a \emph{repelling} periodic cycle. 

For each  $n\ge 1$, let  $\mathrm{Fix}^\ast(f^{\circ n})$ denote the set of all repelling periodic points of period $n$ for $f$.  The \emph{Julia set} $\mathcal{J}(f)$ of $f$ is defined to be the closure of the grand orbit of the repelling periodic points of $f$:  
$$\mathcal{J}(f):=\overline{\mathrm{GO}(\cup_{n\ge 1}\mathrm{Fix}^\ast(f^{\circ n}))}.$$ 
It follows from the definition that  $f(\mathcal{J}(f)) \subseteq \mathcal{J}(f)$, so $f$ induces the sub-dynamical system $(\Jcal(f),f)$. We mention that $\mathcal{J}(f)$ could be empty. 

Given a disk $U$ in $X$, we say that $f$ is a \emph{(distance) scaling map} on $U$ if the following conditions hold:
\begin{enumerate}
\item $f(U)$ is a disk in $X$, and
\item there exists a real number  $\lambda>0$ such that  for any $x,y\in U$  we have: 
$$\rho(f(x),f(y))=\lambda\cdot\rho(x,y).$$
\end{enumerate}
In this case, we call $U$ a \emph{(distance) scaling disk} for $f$ and observe that $f$ is a bijection from $U$ onto $f(U)$. Observe that any trivial disk is a scaling disk.

We say that a point $x\in\mathcal{J}(f)$ is a \emph{regular} point if there exists a disk containing $x$ where $f$ is a scaling map; 
otherwise, it is called \emph{singular}. Denote by $\mathrm{\bf Sing}(f)$ the set of singular points of $f$ in $\mathcal{J}(f)$.

\begin{definition}
Given a non-archimedean dynamical system $(X, f)$, we say that  
\begin{enumerate}
\item $f$ is \emph{hyperbolic} if $\mathrm{\bf Sing}(f)$ is empty;
\item $f$ is \emph{subhyperbolic} if  $\mathrm{\bf Sing}(f)$ is finite and, for any $x\in\mathrm{Sing}(f)$, there is  $k_x>0$ such that $f^{\circ k_x}(x)$ is a repelling periodic point of $f$.
\end{enumerate}
\end{definition}

Trivially, if $f$ is hyperbolic, then it is subhyperbolic. Moreover, if $f$ is a subhyperbolic rational map with compact $\mathcal{J}(f)$, then $\Jcal(f)$ admits some expanding (singular) metric, see \cite{Fan24}.  

In the context of rational maps $\phi\in K(z)$, recall from \S\ref{sec:intro} the set of critical points $\mathrm{Crit}(\phi,K)$ and the Julia set $\mathcal{J}(\phi,K)$. Regarding $\phi$ as a continuous map defined on the non-archimedean space $\mathbb{P}^1(K)$, we immediately have that 
$$\mathrm{Crit}(\phi, K)\cap\mathcal{J}(\phi)=\mathrm{\bf Sing}(\phi).$$
Moreover, if $\phi$ is  subhyperbolic, then $\mathcal{J}(\phi,K)$ coincides to $\mathcal{\mathcal{J}(\phi)}$: 
 
\begin{lemma}\label{lem:J}
Let $\phi\in K(z)$ be a subhyperbolic rational map of degree at least $2$. Then 
$$\mathcal{J}(\phi,K)=\mathcal{\mathcal{J}(\phi)}.$$
In particular, 
$$\mathrm{Crit}(\phi, K)\cap\mathcal{J}(\phi,K)=\mathrm{\bf Sing}(\phi).$$
\end{lemma}
\begin{proof}
Since $\mathcal{F}(\phi,K)$ is forward invariant under $\phi$ by definition, it follows that $\mathcal{J}(\phi,K)$ is backward invariant under $\phi$. 
Noting that $\mathrm{Fix}^\ast(\phi^{\circ n})$ is a subset of $\mathcal{J}(\phi,K)$ for any $n\ge 1$, since $\mathcal{J}(\phi, K)$ is closed, we conclude that $\mathcal{\mathcal{J}(\phi)}\subseteq\mathcal{J}(\phi,K)$. 

Now pick $x\in\mathcal{J}(\phi,K)$ and consider a sufficiently small neighborhood $U$ of $x$ in $\mathbb{P}^1(K)$. By the non-archimedean Montel's Theorem, see \cite[Corollary D]{Favre12} (see also \cite{Hsia00}), there exists $n:=n_U\ge 1$ such that $U$ is a subset of $\phi^{\circ n_U}(U)$. If $\phi^{\circ n}$ is a scaling map on $U$, then $U$ contains a point in $\mathrm{Fix}^\ast(\phi^{\circ n})$. If $\phi^{\circ n}$ is not a scaling map on $U$ for any neighborhood $U$ of $x$, then $x$ is an accumulation point of the iterated preimages of a critical point, say  $c$, in $\mathrm{Crit}(\phi, K)\cap\mathcal{J}(\phi,K)$. Since $\phi$ is subhyperbolic, the point $c$ is contained in $\mathrm{GO}(\cup_{n\ge 1}\mathrm{Fix}^\ast(\phi^{\circ n}))$. It follows that $x\in \mathcal{\mathcal{J}(\phi)}$. Hence $\mathcal{J}(\phi,K)\subseteq\mathcal{\mathcal{J}(\phi)}$. 
\end{proof}


\subsection{Artin-Mazur zeta function}\label{sec:AM}

In this section, we state the Artin-Mazur zeta function in several settings. We refer to \cite{Pollicott02, Pollicott20} for references.

For a continuous map $f: X\to X$ on a non-archimedean space $X$, if $\mathrm{Fix}^*(f^{\circ m})$ is a finite set for each $m\ge 1$,  the Artin-Mazur zeta function on $\mathcal{J}(f)$ is 
 defined by 
$$\zeta_f(X, t):=\exp\left(\sum_{m=1}^\infty\, \#\mathrm{Fix}^*(f^{\circ m})\cdot\frac{t^m}{m}\right).$$
In fact, $\zeta_f(X, t)\in\mathbb{Z}[[t]]$ from the product formula
$$\zeta_f(X, t)=\prod_{\text{repelling cycles}\ \gamma}(1-t^{\#\gamma})^{-1}.$$

The Artin-Mazur zeta function is also well-defined for finite directed graphs. Consider a finite directed graph  $G$. Let $A_G$ be its \emph{adjacency matrix} of $G$, that is,  each entry of $A_G$ represents the number of directed edges of $G$ connecting the corresponding vertices. Then $A_G$ has non-negative integer entries.  The Artin-Mazur zeta function for $G$ is defined by 
$$\zeta(G, t):=\exp\left(\sum_{m=1}^\infty\, \mathcal{N}_m(G)\cdot\frac{t^m}{m}\right),$$
where $\mathcal{N}_m(G)$ is the number of loops consisting of $m$ (not necessarily distinct) directed edges. 
It follows that $\mathcal{N}_m(G)=\trace((A_G)^m)$ and 
$$\zeta(G, t)=\det(I-tA_G)^{-1},$$
where $I$ is the identity matrix of the same size as $A_G$. Let $\lambda\ge 1$ be the leading eigenvalue of $A_G$ (in fact, $\lambda$ is a weak Perron number). Then, as a function in $\mathbb{C}$, the function $\zeta(G,t)$ is non-zero and analytic for $|t|<1/\lambda$ and has a simple pole at $t=1/\lambda$. 

Moreover, given a finite matrix  $A$ with non-negative integer entries, let $G_A$ be a finite directed graph such that $A_G=A$. Then 
$$\zeta(G_A, t)=\det(I-tA)^{-1}.$$

Now for symbolic dynamical systems, consider a subshift of finite type $(\Sigma_B,\sigma)$ with adjacency matrix $B$, where $\sigma$ is the left shift on $\Sigma_B$. The Artin-Mazur zeta function on $\Sigma_A$ is defined by 
$$\zeta_\sigma(\Sigma_B, t):=\exp\left(\sum_{m=1}^\infty\, \#\mathrm{Fix}_B(\sigma^{\circ m})\cdot\frac{t^m}{m}\right),$$
where $\mathrm{Fix}_B(\sigma^{\circ m})$ is the set of the periodic points of $\sigma$ of period $m$ in $\Sigma_B$.  Then 
$$\zeta_\sigma(\Sigma_B, t)=\zeta(G_B, t)=\det(I-tB)^{-1}.$$

Now consider a continuous map $f: X\to X$ on a non-archimedean space $X$ and a subshift of finite type $(\Sigma_B,\sigma)$. If $(\mathcal{J}(f),f)$ is topologically conjugate to $(\Sigma_B,\sigma)$, then 
$$\zeta_f(X, t)=\zeta_\sigma(\Sigma_B, t)=\det(I-tB)^{-1};$$
moreover, the topological entropy $\mathfrak{h}(f)$ for $(\mathcal{J}(f),f)$  and the topological entropy $\mathfrak{h}_B(\sigma)$ for $(\Sigma_B,\sigma)$ satisfy that 
$$\mathfrak{h}(f)=\mathfrak{h}_B(\sigma)=\log\lambda_B=\limsup_{m\rightarrow\infty}\frac{1}{m}\log\#\mathrm{Fix}_B(\sigma^{\circ m})=\limsup_{m\rightarrow\infty}\frac{1}{m}\log\#\mathrm{Fix}^*(f^{\circ m}),$$
 where $\lambda_B\ge 1$ is the leading eigenvalue of $B$.

\section{Rationality of the zeta function}\label{sec:zeta-ra}
In this section, we  assume that $(X,\rho)$ is a complete non-archimedean metric space and let $f:X\to X$ be a subhyperbolic continuous map with nonempty Julia set $\mathcal{J}(f)$. Assuming that the set of singular points $\Sing(f)$ of $f$ is nonempty. Then, by subhyperbolicity, for each singular point $x\in\Sing(f)$ there is a unique repelling cycle $\gamma_x\subset \Ocal_{f}(x) $ where the orbit of $x$ eventually lands. 
Let $\Gamma(f) \subset \Ocal_{f}(\Sing(f))$ be the set of all repelling cycles $\g_{x}$ for $x\in \Sing(f).$ Note that  $\Gamma(f)$ is a finite set since $\mathrm{\bf Sing}(f)$ is finite. 

Our goal is to establish the rationality of the zeta function for $f$ under some assumptions on $\mathcal{J}(f)$ and $\mathrm{Sing}(f)$. 

\subsection{Definitions and main result}\label{sec:result}
First, we introduce some constraints on the behavior of points in the neighborhood of the singular points of $f$ under its action. In the following,  we use the notations $\diam(V)$ to denote the diameter of a disk $V$ and  $\rho(x, V)$ to denote the  the distance between $x$ and $V$, as usual. 
\begin{definition}\label{def:normal}
A point $x\in\mathrm{\bf Sing}(f)$ is \emph{normal} if there exists a disk $U\subseteq X$ containing $x$ and a positive number $r_x\le 1$ such that
\begin{enumerate}
\item $x$ is the unique singular point in $U$, that is, $U\cap \mathrm{\bf Sing}(f)=\{x\}$, 
\item $f(y)\not=f(x)$ for any $y\in U\setminus\{x\}$, 
\item any maximal scaling disk  $V$ in $U\setminus\{x\}$ has diameter at least $r_x\cdot\rho(x,V)$, and 
\item the image $f(V)$ of $V$ has diameter at least $r_x\cdot\rho(f(x), f(V))$. 
\end{enumerate}
\end{definition}

\begin{remark} \label{rem:normal}
We call a disk $U$ that satisfies all the conditions in Definition~\ref{def:normal} a \emph{normal disk} for the singular point $x$. By condition~(3) of the definition,  any disk $V'\subset \left(U \setminus \{x\}\right)$ with $\diam(V')/\rho(x, V') \le r_{x}$ is a scaling disk for $f$. Moreover, condition~(4) ensures that  any disk $V' \subset \left(U \setminus \{x\}\right)$ satisfying $\diam(f(V'))/\rho\left(f(x), f(V')\right) \le r_{x}$ is likewise a scaling disk for $f$.
\end{remark}

In what follows, for each $\gamma\in\Gamma(f)$, we fix a marked point $p_{\gamma}\in\gamma$. For any $x\in \GO(\g)$,   we define $\ell_x\ge 0$ to be the smallest integer $\ell$ such that $f^{\circ \ell}(x)=p_{\gamma}$.  

Let $s\ge 1$ be the  length of the repelling cycle $\g$.  Then  there exists a disk $U$ containing $p_{\g}$ such that:
\begin{itemize}
\item $U$ is scaling for $f^{\circ s}$, and 
\item   $U\subsetneq f^{\circ s}(U)$.
\end{itemize}
Since $f^{\circ s}$ is scaling on $U$ and fixes $p_{\g}$, its image $U' = f^{\circ s}(U)$ is also a disk of $X$ containing $p_{\g}.$ Thus,  the restriction of $f^{\circ s}$ to $U$ gives rise to a (local) homeomorphism 
$f^{\circ s}: U \to U'$ that fixes $p_{\g}.$ Then one can show that the points in $\Jcal(f) \cap U$ exhibit a certain homogeneous property under the action of $f^{\circ s}$.  

Now, consider the pullbacks of $U$ and $U'$ by $f$ along the (truncated) orbit of $x \ne p_{\g}$ 
$$\Ocal_{f}(x, \ell_{x}) := \{x, f(x), \ldots, f^{\circ(\ell_{x} -1)} (x)\}.$$ 
If this orbit contains no singular points, one may naturally expect that a similar phenomenon persists in a neighborhood of each point in $\Ocal_{f}(x, \ell_{x})$. However, this cannot generally be expected near singular points. To address this, we introduce the following condition: 

\begin{definition}\label{def:homogeneous}
We say that $f$ satisfies the \emph{Homogeneous Condition} if there exists 
a integer valued function $\varphi:\Gamma(f)\to\mathbb{Z}_{>0}$ and, 
for each $x\in\mathrm{\bf Sing}(f)$, there exist a disk $V_x\subset X$ containing $x$ and a  scaling map $T_x: V_x\to X$ fixing $x$
 such that  the following conditions hold: 
\begin{enumerate}
\item  $f^{\circ \ell_x}\circ T_x=f^{\circ(\varphi(\gamma_x)+\ell_x)}$ on $V_x$; 
\item if $y:=f^{\circ j}(x)\in\mathrm{Sing}(f)$ for some $j\ge 1$, then 
$f^{\circ j}(V_x)\subseteq V_y$ and 
$$f^{\circ j}\circ T_x=T_y\circ f^{\circ j}\quad \text{on $V_{x}$}.$$
\end{enumerate} 
\end{definition}

\noindent See Figure \ref{f:homo} below for an illustration of Definition \ref{def:homogeneous}. 

 \begin{figure}[h]
     \includegraphics[width=0.5\textwidth]{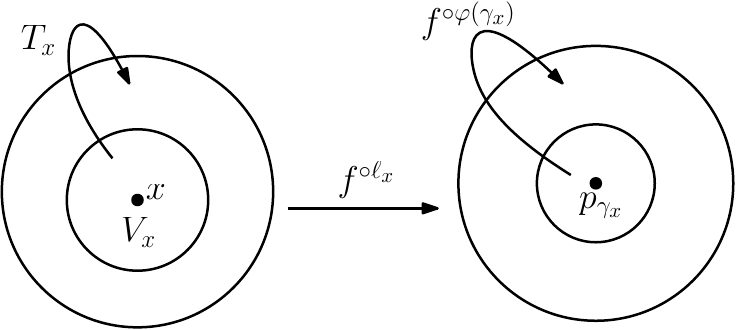}
	\caption{The map $T_x$ in Definition \ref{def:homogeneous}.}
    \label{f:homo}
\end{figure}

\begin{remark}
\begin{enumerate}
\item
Suppose that $f$ satisfies the Homogeneous Condition at a point $x\in\mathrm{\bf Sing}(f)$. Then an appropriate choice of $\varphi(\gamma)$ can be taken to be a multiple of the lengths of  all 
$\gamma\in\Gamma(f)$. 
\item
If $f$ is a subhyperbolic rational map defined over a non-archimedean field of characteristic $0$, then each of its singular points (the critical points in $\mathcal{J}(f)$) is normal and $f$ satisfies  the Homogeneous Condition, see \S\ref{sec:proof-rational}. 
\end{enumerate}
\end{remark}


The main goal in this section is to establish the following rationality result; its proof relies on a series of results which will be given in the following  subsections. 
\begin{theorem}\label{thm:zeta}
Let $f:X\to X$ be a non-archimedean subhyperbolic continuous map. Assume that 
\begin{enumerate} 
\item $\mathcal{J}(f)$ is compact,
\item each point in $\mathrm{Sing}(f)$ is normal, and 
\item $f$ satisfies the homogeneous condition. 
\end{enumerate}
Then for each $n\ge 1$, 
$$\#\mathrm{Fix}^*(f^{\circ n})<+\infty.$$
Moreover, 
$\zeta_f(X,t)=P(t)/Q(t)$, where $P(t), Q(t)$ in $\mathbb{Z}[t]$ are coprime and  $P(t)$ is a product of cyclotomic polynomials.
\end{theorem}
Under the assumptions in Theorem \ref{thm:zeta}, we will describe the dynamics of $f$ on $\Jcal(f)$ with a finite Markov partition, which in fact leads to an exact formula for $\zeta_f(X,t)$, see \S\ref{sec:factoring} and \S\ref{sec:pr}. 

In the following subsections, our goal is to establish Theorem~\ref{thm:zeta}. Throughout this part, we assume that the Julia set $\mathcal{J}(f)$ is a compact subset of $X$.

We begin in \S\ref{sec:dyn} by analyzing the local dynamics along the orbits of singular points, under the assumption that the singular set $\Sing(f)$ is nonempty and satisfies the {\em Normal Condition} and the {\em Homogeneous Condition}. In \S\ref{sec:splitting}, we construct a natural countable Markov partition of $\mathcal{J}(f)$ through a procedure called {\em splitting}. Then, in \S\ref{sec:factoring}, we refine this countable Markov partition into a finite one by applying the Homogeneous Condition. Finally, in \S\ref{sec:pr}, we use the refined Markov partition to complete the proof of Theorem~\ref{thm:zeta}.

\subsection{Dynamics along the orbits of singular points}\label{sec:dyn}
In this subsection, we assume that the Julia set $\mathcal{J}(f)$ is compact, and that the set of singular points $\mathrm{\bf Sing}(f)$ is nonempty. The main purpose of the subsection is to study the local dynamics around the set of the singular points of $f$. We focus on two specific cases: when the set of singular points $\Sing(f)$ satisfies the Normal Condition (see Definition~\ref{def:normal}) and, furthermore, when it satisfies  the Homogeneous Condition (see Definition~\ref{def:homogeneous}). Our goal is to construct appropriate coverings of the Julia set $\Jcal(f)$ such that the dynamical system $(\Jcal(f), f)$ can be encoded in terms of a suitably defined symbolic dynamical system.

\subsubsection{Normal Condition and Markov Property}\label{subsec:normal}
We begin by analyzing the dynamics in neighborhoods along the points of $\Ocal_{f}(\Sing(f))$ under the assumption that the set of singular points $\Sing(f)$ satisfies the Normal Condition.
For $\gamma\in\Gamma(f)$, we set  
$$\mathrm{GO}^*(\gamma):=\bigcup_{x\in \mathrm{GO}(\gamma)\cap\mathrm{\bf Sing}(f)} \Ocal_{f}(x). $$
We will construct, for each $\g \in \Gamma(f)$, a neighborhood of $\mathrm{GO}^*(\gamma)$ which contains, in addition to  $\mathrm{GO}^*(\gamma)$ itself, countably many scaling disks intersecting $\mathcal{J}(f)$: 

\begin{lemma}\label{lem:singular}
For each $\g \in \Gamma(f)$, there exist a neighborhood $\overline W(\gamma)$ of $\mathrm{GO}^*(\gamma)$ which is a disjoint union of finitely many closed disks in $X$ and a collection $\mathcal{U}(\gamma)$ of countably many scaling disks contained in $\overline W(\gamma)$ such that the following hold:
 \begin{enumerate}
\item For each  $U\in\mathcal{U}(\gamma)$, $U\cap\mathcal{J}(f)\not=\emptyset$. 
\item $\bigcup_{U\in\mathcal{U}(\gamma)} U$ covers $\left(\overline W(\gamma)\cap\mathcal{J}(f)\right)\setminus\mathrm{GO}^*(\gamma)$. 
\item For any $U\in\mathcal{U}(\gamma)$, either $f(U)\in\mathcal{U}(\gamma)$ or $f(U)\cap \overline W(\gamma')=\emptyset$ for all $\g' \in \Gamma(f).$  
\end{enumerate}
\end{lemma}

\begin{proof}
Let $\g \in \Gamma(f)$ be given and let $p:=p_\gamma \in \gamma$ be the chosen marked point for $\gamma$. We pick a closed disk $\overline D$ in $X$ containing $p$ and pull back $\overline D$ along the orbits in $\mathrm{GO}^*(\gamma)$. Namely, for each point $x\in \mathrm{GO}^*(\gamma)$, we associate a disk $\overline D_x,$ defined as the largest closed disk containing $x$ such that $f^{\circ \ell_{x}}(\overline D_x)\subseteq \overline D.$ 
Let $p:=p_\gamma \in \gamma$ be the chosen marked point for $\gamma$. We pick a closed disk $\overline D$ in $X$ containing $p$ and pull back $\overline D$ along the orbits in $\mathrm{GO}^*(\gamma)$. Namely, for each point $x\in \mathrm{GO}^*(\gamma)$, we associate a disk $\overline D_x,$ defined as the largest closed disk containing $x$ such that $f^{\circ \ell_{x}}(\overline D_x)\subseteq \overline D.$ 

Let  $s\ge 1$ denote the length of $\gamma$. we can shrink the disk $\overline D$ if necessary so that the following conditions are satisfied: 
\begin{enumerate}
\item The disk $\overline D_{p}=\overline D$ is a scaling disk of $f^{\circ s}$. 
\item The disks $\overline D_x$ for $x\in \mathrm{GO}^*(\gamma)$ are pairwise disjoint. 
\item  For any $y\in\gamma$, if $f^{\circ s}(\overline{D}_y)\cap\overline D_x\not=\emptyset$, then $x=y$. 
\item If $\overline D_x\cap\mathrm{\bf Sing}(f)=\emptyset$, then $\overline D_x$ is a scaling disk for $f$. 
\item If $\overline D_x\cap\mathrm{\bf Sing}(f)\not=\emptyset$ , then $x$ is the unique singular point in $\overline D_x$ and $\overline D_x$ is a normal disk of $x$.
\end{enumerate}
Set 
$$\overline{W}(\gamma):=\bigcup_{x\in\mathrm{GO}^*(\gamma)}\overline D_x,$$
 and observe that for any $y\in\gamma$, the disk $\overline D_y$ is a scaling disk of $f^{\circ s}$.  

For each $z\in\mathrm{GO}^*(\gamma)\cap\mathrm{\bf Sing}(f)$, recall from  Definition \ref{def:normal} the positive number $r_z\le 1$. We fix a  $\delta$  such that 
$$0<\delta<\prod_{z\in\mathrm{GO}^*(\gamma)\cap\mathrm{\bf Sing}(f)}r_z.$$
 In $\overline D_p$, we consider  the following collection of closed disks:
 \begin{equation}\label{eqn:U_p} 
\mathcal{U}_p:=\left\{\text{maximal closed disks}\ U\subseteq \overline D_p\setminus\{p\}: \mathrm{diam}(U)\le\delta\cdot\rho(p,U)\ \text{and}\ U\cap\mathcal{J}(f)\not=\emptyset\right\}.
\end{equation}

For any $x\in\mathrm{GO}^*(\gamma)$, since $f^{\circ \ell_{x}}(\overline D_x)\subseteq \overline D_{p}$ by construction, we consider the preimage $f^{-\ell_{x}}(V)$ of a disk $V\in \mathcal{U}_p$ in $\overline{D}_{x}$. If $f^{-\ell_{x}}(V) \cap \overline{D}_{x}$ is nonempty, then it is a disjoint union of closed disks in  $\overline{D}_{x}.$ Let $U\subseteq \overline D_x\setminus\{x\}$ be one of the maximal closed disks in $f^{-\ell_{x}}(V)\cap \overline{D}_{x}$.
We claim that $f^{\circ \ell_x}$ must be scaling on $U$ and  $f^{\circ \ell_x}(U)=V$. Indeed, since $\mathrm{diam}(V)\le\delta\cdot\rho(p,V)$ and by condition~(4) and (5) given above in the construction of $D_{x}$, we see that $f^{\circ \ell_{x}}$ is certainly scaling on $U$ if the set $W_{x} := \{x, f(x), \ldots, f^{\circ \ell_{x}}(x)\}$ does not contain any singular point. On the other hand, if $W_{x}$ contains a singular point $z$, then since $z$ is normal by assumption and $\delta < r_{z}$, in view of  Definition \ref{def:normal}~(4) and Remark~\ref{rem:normal},  there exists a maximal scaling disk $V'\subseteq \overline D_x$ such that $V'\cap U\not=\emptyset$ and $V\subseteq f^{\circ \ell_x}(V')$,  which implies $U\subseteq V'$ and $f^{\circ \ell_{x}}$ is also scaling on $U$ in this case. If $f^{\circ \ell_{x}}(U)\subsetneq V$, then, since $f^{\circ \ell_{x}}$ is scaling on $V'$, the preimage $f^{-\ell_{x}}(V)$ within $V'$ is a scaling disk strictly larger than $U$, which contradicts the assumption that  $U$ was chosen to be maximal.   This validates the claim. 
We now set 
\begin{equation}\label{eqn:U_x}
\mathcal{U}_x:=\left\{\text{closed disks}\ U\subseteq \overline D_x\setminus\{x\}: f^{\circ \ell_x}(U)\in\mathcal{U}_p\ \text{and}\ U\cap\mathcal{J}(f)\not=\emptyset\right\}.
\end{equation}

Let 
$$\mathcal{U}(\gamma):=\bigcup_{x\in \mathrm{GO}^*(\gamma)}\mathcal{U}_x.$$ 
Then for any $U\in\mathcal{U}(\gamma)$, we have that $U\subseteq \overline{W}(\gamma)$ and $U\cap\mathcal{J}(f)\not=\emptyset$. Moreover, $\mathcal{U}(\g)$ forms a covering of $\left(\bigcup_{x\in \mathrm{GO}^*(\gamma)}\overline D_x\right)\setminus \mathrm{GO}^*(\gamma)$. Thus statements (1) and (2) hold. 

Statement~(3) follows from the construction of $\Ucal(\g)$ and the fact that $p$ is a repelling periodic point. Specifically, the only disks $U\in \Ucal(\g)$ for which $f(U)$ is possibly disjoint from $\overline{W}(\g)$ are those in  $\overline{D}_p\setminus f^{-s}(\overline D_p)$ since  $\overline{D}_p\subsetneqq f^{\circ s}(\overline D_p)$. 

Let $\ell_f$ be a fixed integer greater than $\ell_{x'} + s_{\gamma'}$ for all $x' \in \mathrm{GO}^{\ast}(\gamma')$ and all $\gamma' \in \Gamma(f)$. We choose the closed disk $\overline{D}_p$ to have sufficiently small diameter so that for all integers $0 \le i \le \ell_f$ and all $\gamma' \in \Gamma(f)$ with $\gamma' \ne \gamma$, we have
$$f^{\circ i}(\Dbar_{p})\cap \Dbar_{p_{\g'}} = \emptyset.$$ 
Let $U \in \mathcal{U}(\gamma)$ be a scaling disk such that $f(U) \cap \overline{W}(\gamma) = \emptyset$. By construction, $U \in \mathcal{U}_p$, and hence $U \subset \overline{D}p$. Suppose, for contradiction, that $f(U) \cap \overline{W}(\gamma') \ne \emptyset$ for some $\gamma' \in \Gamma(f)$. Then there exists a point $\beta \in f(U) \cap \overline{W}(\gamma')$, which implies that $\beta \in \overline{D}_{x'}$ for some $x' \in \mathrm{GO}^{\ast}(\gamma')$ and $\beta = f(y)$ for some $y \in \overline{D}_p$.
  
 It follows that 
  $$f^{\circ \ell_{x'}}(\beta)\in \Dbar_{p_{\g'}}$$ 
  and so  
  $$f^{\circ (\ell_{x'} +1)}(y) = f^{\circ \ell_{x'}}(\beta)\in f^{\circ (\ell_{x'} +1)}(\Dbar_{p})\cap \Dbar_{p_{\g'}}.$$ 
 But since $\ell_{x'} + 1 \le \ell_{x'} + s_{\gamma'} \le \ell_f$, this contradicts the choice of $\overline{D}_p$, which was constructed to avoid such intersections.
 
Therefore, we conclude that   $$f(U)\cap \Wbar(\g') = \emptyset$$ 
 for all $\g'\in \Gamma(f)$,  as desired.

To complete the proof, we must show  the collection $\mathcal{U}(\gamma)$ is countable. It is sufficient to show that $\mathcal{U}_x$ is countable for each $x\in\mathrm{GO}^*(\gamma)$. For each large  $n\ge 1$, consider the open disk $D_{x,n}:=D(x,1/n)$. Pick a sufficiently large $n_0\gg 1$ so that $\overline D_{x,n}\subsetneq\overline D_x$,  and consider the (closed) annuli $\overline A_{x,n_0}:=\overline D_x\setminus D_{x,n_0}$ and $\overline A_{x,n+1}=\overline D_{x,n}\setminus D_{x,n+1}$ for $n\ge n_0$. Then for each $n\ge n_0$, the intersection $\overline A_{x,n}\cap\mathcal{J}(f)$ is compact, and hence $\mathcal{U}_x$ contains finitely many elements which are subsets of  $\overline A_{x,n}$. Since the union of  the annuli $\overline A_{x,n}, n\ge n_0$ is $\overline D_x\setminus\{x\}$, we conclude that $\mathcal{U}_x$ is countable. 
\end{proof}
\begin{remark}\label{rem:invariance of Up}
Let $U\in \Ucal_{p}$ be such that $f^{\circ s}(U) \subset \Dbar_{p}$, then the proof of Lemma~\ref{lem:singular} also shows that  $f^{\circ s}(U) \in \Ucal_{p}.$ Indeed, since $U$ satisfies the inequality $\diam(U) \le \d \cdot \r(p, U)$ and $f^{\circ s}$ is scaling on $\Dbar_{p}$, it follows that $f^{\circ s}(U)$ also satisfies the same inequality. By the definition of $\Ucal_{p}$ there exists a scaling disk $V \in \Ucal_{p}$ such that $f^{\circ s}(U) \subseteq V.$ If the inclusion is strict, i.e., $f^{\circ s}(U) \subsetneq V$, then the same reasoning applied to the inverse $f^{-\ell_{x}}$ on $\Dbar_{p}$ implies that $U\subsetneq f^{-\ell_{x}}(V)$, which contradicts the maximality of $U.$  
\end{remark}

Throughout our discussions below, we say that a collection $\mathcal{U}$ of disjoint nontrivial disks and possibly a finite collection  of trivial disks (points) in $X$ forms a {\em covering} of $\mathcal{J}(f)$ if each disk $U\in \Ucal$ intersects $\Jcal(f)$ nontrivially, i.e., $U\cap\Jcal(f) \ne \emptyset$ and if $\Jcal(f)$ is  contained in the union of these disks:   
$$\Jcal(f)\subseteq \bigcup_{U\in \Ucal} U.$$ 
If, furthermore, $\Ucal$ consists of scaling disks for $f$, we say that $\Ucal$ is a covering of $\Jcal(f)$ by scaling disks.  Moreover, we say that $\Ucal$ satisfies the \emph{Markov property} (with respect to $f$), if for every $U\in \Ucal$, its image  $f(U)$ contains at least a disk in $\mathcal{U}.$ This covering induces a symbolic dynamical system that encodes the dynamics of $f$ in $\mathcal{J}(f)$.

By Lemma~\ref{lem:singular},  we know that the Julia set in a neighborhood of $\Sing(f)$ can be covered by countably many scaling disks. Moreover, this covering nearly fulfills the conditions required for the Markov property. Our goal is to demonstrate that, through suitable modifications, this initial covering can be refined into one that fully satisfies the Markov property. The main tool for achieving this refinement is a process, called {\em splitting}, which we define as follows.

\begin{definition}\label{def:splitdisc}
Let $Y$ be a nonempty subset of $X$. For a disk $U$ of $X$ together with a proper (possibly trivial)  subdisk 
$U_0\subsetneq U$ such that $U_{0}\cap Y \ne \emptyset$, we say that  the following collection of disks
$$\left\{U_0\right\} \cup \left\{V: V\ \text{is a  maximal disk in}\ U\backslash U_0\ \text{intersecting}\ Y\right\}$$
is the  {\em splitting} of $U$ via $U_0$ with respect to $Y$. 
\end{definition}

Let $\psi: X \to X$ be a continuous map with nonempty Julia set $\mathcal{J}(\psi)$ and let $\mathcal{U}$ be a covering of $\Jcal(f)$ by scaling disks for $f$. Assume that $\Ucal$ contains at least one nontrivial  disk. We claim that for each nontrivial disk $U\in \Ucal$, there exists a positive integer $m$ such that $\psi^{\circ m}(U)$ contains a disk in $\Ucal$. 

To prove the claim, we note that since each disk  $U \in \Ucal$ has a nontrivial intersection with $\Jcal(f)$, its image $\psi(U)$, which remains a disk under $\psi$, must intersect some $U'\in \Ucal$. 
Suppose, for the sake of contradiction, that the claim does not hold for $U$. That is, $\psi^{\circ n}(U)$ remains properly contained in a disk in $\Ucal$ for every positive integer $n$. Also note that in this case,  $U$ is a scaling disk for $\psi^{\circ n}$ for all $n$. By considering a finite forward orbit of $U$ under $\psi$ if necessary, we may assume that $U$ contains a repelling periodic point, say $p$,  of period $s$. It follows that 
$$p\in U\subsetneq\psi^{\circ s}(U).$$  
This contradicts the assumption on $U$ and hence the claim follows.

Now we define $m_\psi(U)\ge 1$ as the least positive integer satisfying the assertion of the claim for a nontrivial disk $U\in \Ucal$, that is, $\psi^{\circ m_\psi(U)}(U)$ contains at least one disk $U'\in \Ucal.$    Note that $\Ucal$ satisfies the Markov property if and only if $m_\psi(U)=1$ for all nontrivial disk $U\in \Ucal.$  Suppose that $m_\psi(U)= 1$  for all but finitely many nontrivial disks  $U\in \Ucal$.  We then introduce the index 
$$M_\psi(\mathcal{U}):=\sum_{\text{nontrivial}\ U\in\mathcal{U}}(m_\psi(U)-1).$$
In the following, we show that the index $M_\psi(\mathcal{U})$ serves as an indicator of the number of splitting steps required to refine $\Ucal$ into a covering of $\Jcal(\psi)$ by scaling disks that satisfies the  Markov property.

\begin{lemma}\label{lem:split}
Let $\psi:X\to X$ be a continuous map with nonempty Julia set $\mathcal{J}(\psi)$ and let $\mathcal{U}$  be a covering of $\Jcal(\psi)$ by scaling disks. Suppose that $m_{\psi}(U) = 1$ for all but finitely many nontrivial disk $U\in \Ucal.$ Then, there exists a refinement of $\mathcal{U}$ that satisfies the Markov property.
\end{lemma}

\begin{proof}
Suppose that there exists a nontrivial disk $U\in \Ucal$ with $m_{\psi}(U) > 1.$ For ease of notation, let $m = m_{\psi}(U).$ Observe that $U$ is a scaling disk for $\psi^{\circ (m-1)}.$ 

Let $U'\in\mathcal{U}$ be the nontrivial disk such that $\psi^{\circ (m-1)}(U)\subsetneq U'$. Then,  
splitting $U'$ via $\psi^{\circ(m-1)}(U)$ with respect to $\mathcal{J}(\psi)$ yields  a refinement of $\mathcal{U}$, denoted by $\mathcal{U}^{\ast}$. We claim  that 
$$M_\psi(\mathcal{U}^{\ast})=M_\psi(\mathcal{U})-1.$$ 

Observe that we must have that $\psi^{\circ(m-1)}(U)\in \Ucal^{\ast}$ and $m_{\psi}(U) = m-1$ in $\Ucal^{\ast}$, since 
$$\psi\left(\psi^{\circ(m-1)}(U)\right) = \psi^{\circ m}(U)$$ 
contains a disk $U^{''}$ in $\Ucal$. As both $U'$ and $U^{''}$ are disks in $\Ucal$, they are either equal or disjoint. In either case, we have that   $\psi^{\circ m}(U)$  contains a disk in $\Ucal^{\ast}.$ Indeed, if 
$U'  = U''\subset \psi^{\circ m}(U)$, then the splitting of $U'$ contains a disk in  in $\Ucal^{\ast}$; if $U' \ne U''$ then the disk $U''$  is just a disk belonging to $\Ucal^{\ast}.$

To prove our claim, it suffices to show that for every maximal disk $V\subseteq U' \setminus\psi^{\circ (m-1)}(U)$ which intersects $\Jcal(\psi)$ nontrivially, $\psi(V)$ contains a disk in $\mathcal{U}^{\ast}$.  This follows because  the difference between $\Ucal^{\ast}$ and $\Ucal$ arises precisely from the splitting of $U'$ via $\psi^{\circ(m-1)}(U)$. 

Suppose, for contradiction, that there exists such a $V$ for which the claim fails. Then $\psi(V)\subsetneq V'$ for some $V' \in \Ucal^{\ast}.$  Then we must have $V'\subseteq \psi(U')$ for otherwise $\psi^{\circ m}(U)\subset \psi(U')\subseteq V'$. It follows that $\psi^{\circ m}(U)$ is contained in the disk $V'$ in $\Ucal^{\ast}$ which is not the case.  
Noting that  $\psi$ is scaling in $U'$, we conclude that $(\psi|_{U'})^{-1}(V')$ is contained in a maximal disk in $U'\setminus\psi^{\circ (m-1)}(U)$, implying $(\psi|_{U'})^{-1}(V')\subseteq V$, and thus $V'\subseteq\phi(V)$, a contradiction. 

By repeatedly applying this splitting process, we eventually obtain a refinemen $\tilde{\Ucal}$ of $\Ucal$ in which $m_{\psi}(U)=1$ for every nontrivial disk $U\in \tilde{\Ucal}$. Equivalently, the refined covering $\tilde{\Ucal}$ satisfies the Markov property as desired.
\end{proof}

Combining Lemmas~\ref{lem:singular} and \ref{lem:split}, we are able to construct a covering of the Julia set $\Jcal(f)$ that satisfies the Markov property. 

\begin{lemma}\label{lem:markov}
The Julia set $\Jcal(f)$ admits a covering consisting of at most countably many scaling disks for $f$, and this covering can be constructed to satisfy the Markov property. 
\end{lemma} 

\begin{proof}
We first note that if  the set $\Sing(f)$ of singular points is empty, then we may start with a covering $\Ucal$ of the Julia set $\Jcal(f)$ consisting of finitely many non-trivial scaling disks for $f$ since $\Jcal(f)$  is compact. Then, by Lemma~\ref{lem:split}, after a finite steps of splittings, we arrive at a refinement of $\Ucal$ which also consists of finitely many scaling disks and it satisfies the Markov property.

In the following, we assume that $\Sing(f)$ is nonempty and write $\Gamma(f):=\{\gamma_1,\dots,\gamma_{i_0}\}$.  For each $1\le i\le i_0$, we let $\mathcal{U}(\gamma_i)$ be the collections of scaling disks contained in $W(\g_{i})$ as constructed  in Lemma \ref{lem:singular} and put  
$$\mathcal{U}:=\bigcup_{i=1}^{i_0}\left(\mathcal{U}(\gamma_i)\cup\left\{\{x\}:x\in\mathrm{GO}^*(\gamma_i)\right\}\right)\quad  \text{and} \quad \Wbar\left(\Gamma(f)\right) := \bigcup_{i=1}^{i_0} \Wbar(\g_{ i }).$$
Then $\mathcal{U}$ is a covering of $W(\Gamma(f))\cap \Jcal(f)$ by scaling disks of $f$ and $\Wbar\left(\Gamma(f)\right)$ is a neighborhood of $\Gamma(f)$ by disjoint union of closed disks in $X$. For the subset $\Jcal(f)^{\ast}:=\Jcal(f)\setminus W(\Gamma(f))$ of the Julia set outside of $W(\Gamma(f))$, there exists a covering $\mathcal{U}'$ of $\Jcal(f)^{\ast}$ by finitely many nontrivial scaling disks since  $\mathcal{J}(f)$ is compact.  Then the collection 
$$\mathcal{U}^{''}:=\mathcal{U}\cup\mathcal{U}'$$ 
forms a covering of  $\mathcal{J}(f)$ by  scaling disks for $f$. Moreover, it follows from the construction of $\Ucal(\g_{i})$ for $1\le i \le i_{0}$ (see Lemma \ref{lem:singular}~(3))  and the finiteness of $\Ucal'$ that, under the action of $f$,  all  but finitely many nontrivial disks in $\mathcal{U}^{''}$ have images containing at least a disk in $\mathcal{U}_0$. Thus by Lemma \ref{lem:split}, there exists a refinement  $\mathcal{U}_\ast$ of $\mathcal{U}^{''}$ which is covering of $\Jcal(f)$ by scaling disks for $f$  and satisfies the Markov property. Moreover, since $\mathcal{J}(f)$ is compact, we conclude that $\mathcal{U}_\ast$  contains countably many disks. 
\end{proof}

\noindent As a consequence of our construction of the covering of the Julia set, we have the following:
\begin{corollary}\label{cor:finite preimage}
For each $x\in \Ocal_{f}(\Sing(f))$, its preimage $f^{-1}(x)$ under $f$ is finite. 
\end{corollary}

\subsubsection{Homogeneous Condition and $T$-equivalence}\label{subsec:homo}
We continue our investigation of the local dynamics near the singular points of $f.$ Assuming that the set $\Sing(f)$ satisfies the Homogeneous Condition, we demonstrate  that this  condition can be extended to the forward orbit of $\mathrm{\bf Sing}(f)$ in the following sense.

\begin{lemma}\label{lem:T}
 Let $\gamma\in\Gamma(f)$ and denote by $p:=p_\gamma$ the marked point associated to $\gamma$. Then for each $x\in\mathrm{GO}^*(\gamma)$, there exist a disk $U_x$ containing $x$ and a scaling map $T_x: U_x\to X$ fixing $x$  such that the following conditions hold:
\begin{enumerate}
    \item 
     $$
     f^{\circ \ell_x}(U_x)\subseteq U_p \quad \text{and}\quad \mathrm{diam}(f^{\circ\ell_x}(U_x))=\mathrm{diam}(U_p).
     $$
     
    \item Suppose $x_1,x_2\in\mathrm{GO}^*(\gamma)$ are  distinct  and satisfy  $x_2=f^{\circ j_0}(x_1)$ for the  smallest integer $j_0\ge 1$ with  $ f^{\circ j}(x_1)\ne p$ for all $0\le j\le j_0-1$. Then: 
    
\begin{enumerate}
\item
$$
f^{\circ j_0}(U_{x_1})\subseteq U_{x_2}\quad \text{and}\quad\mathrm{diam}(f^{\circ j_0}(U_{x_1}))=\mathrm{diam}(U_{x_2}).
$$

\item The following commutativity holds on $U_{x_1}$: 
$$f^{\circ j_0}\circ T_{x_1}=T_{x_2}\circ f^{\circ j_0}.$$
\end{enumerate}
\end{enumerate}
\end{lemma}

\begin{proof}
We construct the disks $U_{x}$ and maps $T_{x}$ for each $x\in\mathrm{GO}^*(\gamma)$ inductively. We begin with the marked point $p$. Let $U_{p}$ be any disk containing $p$ such that $f^{\circ \varphi(\gamma)}$ is scaling on $U_{p}$, where $\varphi: \Gamma(f) \to \ZZ_{>0}$ is the integer-valued function defined in Definition~\ref{def:homogeneous}.  Define $T_p:U_p\to X$ by 
$$T_p:=f^{\circ \varphi(\gamma)}.$$
It is then clear  that the disk  $U_{p}$ and the map $T_{p}$ satisfy both Conditions (1) and (2) stated in the the lemma.

Now, suppose that for some $x\in\mathrm{GO}^*(\gamma)\setminus \{p\}$, we have already constructed a disk $U_{y}$ and a scaling map $T_{y}$ for $y = f(x)$ that satisfy Conditions~(1) and (2) given in the statement of the lemma.
We construct a disk $U_{x}$ containing $x$  as follows. 

\noindent {\it Case~(I)}: if $x$ is a regular point, then we let $U_x\subseteq X$ be the maximal disk contained in $f^{\circ (-1)}(U_y)$ that contains $x$ and is a scaling disk for $f$. Similarly, let $U'_x\subseteq X$ be the maximal disk contained in $f^{\circ (-1)}(T_{y}(U_y))$ that contains $x$ and is a scaling disks for $f$. 

\noindent {\it Case~(II)}: if $x$ is a singular point, we define $U_{x}$ and $U_{x}'$  as the maximal disks  contained in $f^{\circ (-1)}(U_y)\cap V_{x}$ and $f^{\circ (-1)}(T_{y}(U_y)) \cap V_{x}$,  respectively, both  containing $x$, where $V_{x}$ is the disk associated to $x$ given in  Definition~\ref{def:homogeneous}.

In either case,  we  adjust $U_{y}$, if necessary, so that $\diam(f(U_{x})) = \diam(U_{y})$.  
Next, we define $T_x:=(f\mid_{U'_x})^{-1}\circ T_{x_1}\circ f$ in Case~(I). In Case~(II), we  take $T_{x}$ to be the one already defined on $V_{x}$ as  in   Definition~\ref{def:homogeneous}. By construction, the following identity holds in $U_{x}$:
$$ f \circ T_{x} = T_{y} \circ f $$ 
To ensure that the diameter equalities described in Condition~(2)~(a) are preserved, we note that adjusting 
$U_{y}$ satisfy $\diam(f(U_{x})) = \diam(U_{y})$ requires simultaneously adjusting all disks $U_{y'}$ corresponding to points $y'$ in the truncated orbits $\Ocal_{f}(y, \ell_{y}) = \{y, f(y), \ldots, f^{\circ \ell_{y}}(y) \}.$ 

Starting with $p$ and recursively applying the above process, one can verify that the constructed disks $U_{x}$ and  maps $T_{x}$  satisfy the desired properties outlined in Condition~(1) and Condition~(2). We omit the straightforward verification here. 
\end{proof}

\begin{remark}\label{rem:T}
The proof of Lemma~\ref{lem:T} demonstrates that by appropriately adjusting the size (i.e., the diameter) of the disk $U_{x}$ for each $x\in \Sing(f)$ or $U_{p_{\g}}$ for the marked points $p_{\g}$ for $\g\in \Gamma(f)$, one can ensure that the diameters of the associated disks for all points in the forward orbit $\Ocal_{f}(\Sing(f))$ are as small as desired. 
\end{remark}

In addition to the Homogeneous Condition, assume that $\Sing(f)$ also satisfies the Normal Condition. By Lemma~\ref{lem:markov}, there exists a covering $\Ucal_{\ast}$ of $\Jcal(f)$ by scaling disks that satisfies the Markov property, which can be taken as a refinement of the coverings constructed in Lemma~\ref{lem:singular}. We will show below that the scaling maps $T_{x}$ for $x\in \Ocal_{f}(\Sing(f))$ defined in Lemma~\ref{lem:T} induce an action on all but finitely many  scaling disks in $\Ucal_{\ast}.$

Recall from the proof of Lemma~\ref{lem:singular} that for each $x\in \Ocal_{f}(\Sing(f))$,  there exists a disjoint closed disk $\Dbar_{x}$ containing $x$ that satisfies  certain natural conditions: namely, $\Dbar_{x}$ is a scaling disk for $f$ when $x$ is regular, and  a normal disk for $x$ if $x$ is singular. Moreover,  for each such $x$, a covering $\Ucal_{x}$ of $\Dbar_{x} \cap \Jcal(f)$ by scaling disks is constructed (see~\eqref{eqn:U_p} and \eqref{eqn:U_x}).

Let us retain the notation from Lemma~\ref{lem:T}, where $U_x$ denotes the disk domain of the scaling map $T_x$ for each $x \in \Ocal_{f}(\Sing(f))$. To prepare for analyzing the action of $T_{x}$ on scaling disks in $\Ucal_{\ast}$, 
we may, by Remark~\ref{rem:T}, assume that 
\begin{align}
 U_{x} & \subseteq D_{x}\quad \text{for $x\in \Ocal_{f}(\Sing(f))$, and } \label{eqn:Ux}\\
 U_{p_{\g}} & \subseteq f^{-\varphi(\g)}(\Dbar_{p_{\g}}) \quad\text{when $x = p_{\g}$ for $\g\in \Gamma(f)$}. \label{eqn:Up}
\end{align}
We adopt this assumption in the following and define the subcollection
 \begin{equation}\label{eqn:Ucalx}
 \Ucal_{x}^{\ast} = \{U \in \Ucal_{\ast} \mid U\subseteq U_{x} \}.
 \end{equation}
 Recall that the collection  $\Ucal_{x}$ of scaling disks was constructed in Lemma~\ref{lem:singular} for each $x\in \Ocal_{f}(\Sing(f))$.  The covering  $\Ucal_{\ast}$ is obtained by refining the covering $\Ucal'' = \Ucal \cup \Ucal'$ where $\Ucal$ is the union of $\Ucal_{x}$ for all $x\in \Ocal_{f}(\Sing(f))$ and $\Ucal'$ consists of finitely many scaling disks outside $\Wbar(\Gamma(f))$ as demonstrated in Lemma~\ref{lem:markov}, specifically, by splitting  finitely many scaling disks from those in $\Ucal_{p_{\gamma}}$ for $\gamma \in \Gamma(f)$ and possibly from disks located outside the neighborhood $\Wbar(\Gamma(f))$. More precisely, each $U \in \Ucal_{\ast}\setminus \Ucal''$ either satisfies $U \subset \left(\Dbar_{p_{\gamma}} \setminus f^{-s_{\g}}(\Dbar_{p_{\gamma}})\right)$ for some $\gamma \in \Gamma(f)$ of length $s_{\g}$, or $U$ is disjoint from $\Wbar(\Gamma(f))$ (see Lemma~\ref{lem:singular}~(3)).

\begin{lemma}\label{lem:T-equiv}
Let $x\in \Ocal_{f}(\Sing(f))$ and $U\in \Ucal_{x}^{\ast}.$ If  $ T_{x}(U) \subseteq U_{x}$ then $T_{x}(U) \in \Ucal_{x}^{\ast}.$ 
\end{lemma}

\begin{proof}
We first verify this property in the case  where $x$ is the  marked point  $p =  p_{\g}$  for  $\g \in \Gamma(f).$ In this case,  $T_{p} = f^{\circ \varphi(\g)}$. Since $U' = T_p(U) \subset U_p$ by assumption,  we have that $U' \subset  f^{-\varphi(\g)}(\Dbar_{p}) \subseteq  f^{-s_{\g}}(\Dbar_{p}).$ By repeatedly applying Lemma~\ref{lem:singular}~(3), we conclude that $U'$ is a member of $\mathcal{U}_{\ast}$ and thus $T_{p}(U)\in \Ucal_{p}^{\ast}$.

Now let $x \in \Ocal_f(\Sing(f))$ be such that $x \ne p_{\gamma}$ for any $\gamma \in \Gamma(f)$. In this case, as noted above, we have $\Ucal_x \subset \Ucal_{\ast}$. Hence, for $U \in \Ucal_x^{\ast}\subseteq \Ucal_{x}$, the definition of $\Ucal_x$ (see~\eqref{eqn:U_x}) implies that $U'' = f^{\circ \ell_x}(U) \in \Ucal_p$ for the  marked point $p = p_\gamma$ of the repelling cycle $\g \subseteq \Ocal_f(x)$.

Now $U' = T_x(U) \subseteq U_x$ by assumption, we have 
$$
T_{p}(U'') = T_{p}(f^{\circ \ell_{x}}(U)) = f^{\circ \ell_{x}}(T_{x}(U)) = f^{\circ \ell_{x}}(U') \subseteq  f^{\circ \ell_{x}}(U_{x}) \subseteq  U_{p}. 
$$
From the case of marked points already considered, we conclude that $f^{\circ \ell_x}(U') = T_p(U'') \in \Ucal_p^{\ast}$. Moreover, since $f^{\circ \varphi(\gamma)}(T_p(U'')) \subset f^{\circ \varphi(\gamma)}(U_{p})\subset \overline{D}_{p}$, it follows that $T_p(U'') \in \Ucal{\ast} \cap \Ucal_p$. Therefore, by the definition of $\Ucal_x$ (see~\eqref{eqn:U_x}), we conclude that $U' = T_x(U) \in \Ucal_p^{\ast}$, as desired.
\end{proof} 

Lemma~\ref{lem:T-equiv} states that the scaling maps $T_{x}$ act on the subcollection of scaling disks in $\Ucal_{x}^{\ast}$ that are contained in $T_{x}^{-1}(U_{x})$, for each $x \in \Ocal_{f}(\Sing(f))$. Note that there are only finitely many scaling disks in $\Ucal_{\ast}$ that lie entirely outside the union $\bigcup_{x \in \Ocal_{f}(\Sing(f))} T_{x}^{-1}(U_{x})$, since the complement 
$$\Jcal(f) \setminus \bigcup_{x\in \Ocal_{f}(\Sing(f))}\,T_{x}^{-1}(U_{x})$$
is compact. \\
We are now in a position to define an equivalence relation on the collection $\Ucal_{\ast}$ of scaling disks, using the scaling maps $T_{x}$ for $x \in \Ocal_{f}(\Sing(f))$, as follows. 

\begin{definition}\label{def:T}
For two (not necessary distinct) disks $U$ and $U'$ in $\mathcal{U}_\ast$, we say that $U$ and $U'$ are  \emph{$T$-equivalent} if one of the following holds:
\begin{enumerate}
\item both $U$ and $U'$ are contained in $U_x$ for some $x\in\Ocal_{f}(\Sing(f))$ and $U=(T_x)^{\circ m}(U')$ for some $m\in\mathbb{Z}$; or
\item $U=U'$  otherwise.
\end{enumerate}
\end{definition}

By Lemma~\ref{lem:T-equiv}, the relation defined in Definition~\ref{def:T} is indeed an equivalence relation on the collection of scaling disks in $\Ucal_{\ast}$.  For convenience, we will refer to an equivalence class under the $T$-equivalence relation defined above as a \emph{T-class} in what follows. 

Note that if a disk $U\in \Ucal_{\ast}$ is either disjoint from all the disks $U_x$ for $x \in \Ocal_{f}(\Sing(f))$, or if it is a trivial disk associated to some $x \in \Ocal_{f}(\Sing(f))$, then the $T$-class containing $U$ consists of $U$ alone. In other words, $U$ does not share its equivalence class with any other disk. In the following, we refer to such a $T$-class as a \emph{uniclass}. Clearly, if $\Sing(f)$ is empty, then every $T$-class in $\Ucal_{\ast}$ is a uniclass. 

It follows from the definition that if $\Ucal_{\ast}$ has a non-uniclass, then it must be a subcollection of nontrivial scaling disks in $\Ucal_{x}^{\ast}$ for some $x\in \Ocal_{f}(\Sing(f))$. Note that $\Ucal_{x}^{\ast}$ has a unique uniclass, namely the class with $\{x\}$ alone. Moreover, Lemma~\ref{lem:T-equiv} implies that each non-uniclass $\Ccal \subset \Ucal_{x}$  contains a unique member $U_{\Ccal}$, called the terminal member of $\Ccal$, such that $T_{x}(U_{\Ccal})$ is disjoint from $U_{x}.$ 

We now show that there are only finitely many $T$-classes in $\Ucal_{\ast}.$ 

\begin{lemma}\label{lem:finite T-class}
There are  finitely many $T$-classes in $\Ucal_{\ast}.$
\end{lemma}

\begin{proof}
First, observe that since only finitely many scaling disks in $\Ucal_{\ast}$ are disjoint from the union $\bigcup_{x \in \Ocal_{f}(\Sing(f))} T_{x}^{-1}(U_{x})$, it suffices to show that, for each $x \in \Ocal_{f}(\Sing(f))$, the subcollection $\Ucal_{x}^{\ast}$ of $\Ucal_{\ast}$ decomposes into finitely many $T$-classes.

For each such $x$, we have the nested inclusions 
$$T_{x}^{-(i+1)}(U_{x}) \subsetneq T_{x}^{-i}(U_{x}) \quad \text{for $i \ge 0.$} $$ 
Define the annular regions  
$$A_{i} := T_{x}^{-i}(U_{x})\setminus  T_{x}^{-(i+1)}(U_{x})\quad \text{for $i\ge 0$.} $$ 
Then we have the decomposition 
$$U_{x} = \left(\bigcup_{i\ge 0} A_{i}\right)\bigcup \{x\}$$ 
and for each $i\ge 1$, the map  $T_{x} : A_{i} \to A_{i-1}$ is a homeomorphism. 
   
Since $\mathcal{J}(f) \cap A_i$ is compact for each $i \ge 0$, it intersects only finitely many scaling disks in $\mathcal{U}_{x}^{\ast}$. Let $\mathcal{A}_i$ denote the finite collection of scaling disks contained in $A_i$ for each $i \ge 0$. It follows from Lemma~\ref{lem:T-equiv} that the scaling map $T_x$ induces a bijection $\mathcal{A}_{i} \simeq \mathcal{A}_{i-1}$ for all $i \ge 1$. Since 
$$ \Ucal_{x}^{\ast} = \left(\bigcup_{i\ge 0}\, \Acal_{i}\right) \bigcup \{x\},$$
we conclude that $\mathcal{U}_{x}^{\ast}$ is decomposed into finitely many $T$-equivalence classes. As $\Ocal_{f}(\Sing(f))$ is finite, the union $\bigcup_{x \in \Ocal_{f}(\Sing(f))} \mathcal{U}_{x}^{\ast}$ consists of only finitely many $T$-classes. The lemma now follows.
\end{proof}

By Lemma~\ref{lem:T}, for each $x \in \mathcal{O}_{f}(\Sing(f))$, the scaling map $T_x$ commutes locally with $f$, in the sense that
$$f\circ T_{x} = T_{f(x)} \circ f\quad \text{on $U_{x}$.}$$
Since the covering $\mathcal{U}_{\ast}$ satisfies the Markov property with respect to $f$, it is natural to ask whether this property can be descended  to the equivalence class $\mathcal{T}$ under the action of $f$.

This turns out to be possible, provided that the sizes of the disks $U_x$ are carefully adjusted for each $x \in \mathcal{O}_{f}(\Sing(f))$. Let $\mathcal{T}$ denote the collection  of all $T$-classes of scaling disks in $\mathcal{U}{\ast}$. We now describe how such an adjustment can be made to ensure that the induced map on $\mathcal{T}$ also respects the Markov property.

Let $x \in \mathcal{O}_{f}(\Sing(f))$ be given, and suppose that its set of preimages $f^{-1}(x)$ under $f$  is nonempty. Let $V \in \mathcal{U}_{\ast}$ be a scaling disk such that $V \cap f^{\circ (-1)}(x) \ne \emptyset$. It follows that $f(V) \cap U_x \ne \emptyset$. To ensure  the action of $f$ on $T$-equivalence classes is well-defined, we adjust the size of the disk $U_x$, if necessary, so that 
$$U_{x}\subseteq f(V).$$ 
By Corollary~\ref{cor:finite preimage}, there are only finitely many such scaling disks $V$ for which $x \in f(V)$. Therefore, we may choose $U_x$ to be  contained in the intersection of all such images $f(V)$:
$$ U_{x} \subseteq \bigcap_{V\in \Ucal_{\ast} \   x\in f(V)} \; f(V). $$ 
This refinement guarantees that the dynamics of $f$ respect the Markov structure not only on $\mathcal{U}_{\ast}$ but also on $T$-classes.

We continue to denote by $\mathcal{T}$ the collection of $T$-classes with respect to the adjusted closed disks $U_x$, for $x \in \Ocal_f(\Sing(f))$, as constructed above. The map $f$ then induces a well-defined set-valued map 
$$\mathfrak{F}:\mathcal{T}\to 2^{\mathcal{T}}$$
where $[V] \in \mathfrak{F}([U])$ if and only if there exist scaling disks $U', V' \in \mathcal{U}_\ast$ such that $U' \in [U]$, $V' \in [V]$, and 
$$V'\subseteq f(U').$$ 
Here, $[U]$ denotes the T-class containing $U.$  

\begin{remark}\label{rem:acting on T-class}
To illustrate how the map $\mathfrak{F}$ acts on the set of $T$-classes $\mathcal{T}$, let us denote by $\mathcal{C}_{T}(\mathcal{U}_{x}^{\ast})$ the $T$-classes in $\mathcal{U}_{x}^{\ast}$ for each $x \in \mathcal{O}_{f}(\Sing(f))$. It's clear that there's a unique uniclass in each $\Ccal_{T}(\Ucal_{x})$, namely, the class $[\{x\}]$ containing the singleton $\{x\}$ and therefore, $\gF([\{x\}]) = [\{f(x)\}].$ Let $\mathcal{C} \in \mathcal{C}_{T}(\mathcal{U}_{x}^{\ast})$ be a  non-uniclass. Clearly, there are scaling disks $U \in \Ucal_{x}^{\ast}$ such that $f(U)\subset U_{f(x)}$ and thus $\gF(\Ccal)$ contains at least one non-uniclass $\Ccal'\in \Ccal_{T}(\Ucal_{f(p)}).$
We claim that   $\mathcal{C}'$  is the unique non-uniclass belonging to $ \gF(\Ccal) .$

To prove the claim, let's first consider  the case where $x$ is not the marked point $p_{\gamma}$ associated with the (unique) repelling cycle $\gamma$ contained in $\mathcal{O}_{f}(x)$. Then, by Lemma~\ref{lem:T}, we have $f(U) \subset U_{f(x)}$ for each scaling disk $U\in \Ccal$ and $T_{f(x)}(f(U)) = f(T_{x}(U))$ when $U$ is not the terminal member $U_{\mathcal{C}}$ of $\mathcal{C}$. 
Moreover, for the terminal member $U_{\mathcal{C}}$ we have 
$$f(U_{\Ccal}) = U_{\Ccal'} \; \text{where}\; \Ccal' =[f(U)], $$
and consequently  $\gF(\Ccal) = \{\Ccal'\}$.  

Now consider the case where $x$ is the marked point $p = p_{\gamma}$. In this case, we must have 
$$f(U_{p}) \supsetneq U_{f(p)};$$ 
otherwise, we would obtain $f^{\circ s}(U_{p}) \subseteq U_{p}$, where $s$ is the length of $\g$, contradicting the assumption that $p$ is a repelling periodic point. 
\\
Suppose  there exists a member $U \in \mathcal{C}$ such that $f(U)$ is disjoint from $U_{f(p)}$. We claim that any scaling disk $U' \in \mathcal{U}_{\ast}$ with $U' \subset f(U)$ belongs to a uniclass. To prove this,  we consider two cases: 
\smallskip\\
{\bf Case (i)}  $f(U) \notin \mathcal{U}(\gamma)$, where $\mathcal{U}(\gamma) := \bigcup_{x \in \mathrm{GO}^{\ast}(\gamma)} \mathcal{U}_{x}$. \\
By Lemma~\ref{lem:singular}~(3),  $f(U)$ is disjoint from $\overline{\mathcal{W}}(\gamma')$ for all $\gamma' \in \Gamma(f)$.  
\\
{\bf Case (ii)} $f(U) \in \mathcal{U}(\gamma)$.\\
In this case,  we must have $f(U) \subset \overline{D}_{f(p)} \setminus U_{f(p)}.$
\smallskip \\
Now, the claim follows immediately from the definition of a $T$-class in both cases.

Let $V\in \Ccal$ such that $f(V) \subset U_{f(p)}.$ If $p$ is a fixed point for $f$, i.e. $f(p)=p$, then by Remark~\ref{rem:invariance of Up} we have that $f(V)\in \Ucal_{p}$. Furthermore, by the choice~\eqref{eqn:Up} of $U_{p}$, we conclude that in fact 
$f(V)\in \Ucal_{\ast}$ and thus in $\Ucal_{p}^{\ast}$. This shows that $\Ccal' = [f(V)].$ 

Next, let's assume that $s > 1$ and suppose  that there exist two scaling disks $W_{1}, W_{2}$ in $\Ucal_{f(p)}^{\ast}$ such that $W_{1}, W_{2} \subset f(V).$  Since  $f^{\circ (s-1)}(U_{f(p)}) = U_{p}$ by the construction of $U_{x}$ (see Lemma~\ref{lem:T}),  we have that  $f^{\circ s}(V)\subset U_{p} \subset \Dbar_{p}.$ It follows that $f^{\circ s}(V)$ must simultaneously contain at least two elements in $\Ucal_{p}^{\ast}$ which is impossible since $f^{\circ s}(V)\in \Ucal_{p}^{\ast}$ by Remark~\ref{rem:invariance of Up}. In conclusion, we see that $f(V)$ belongs to $\Ucal_{f(p)}^{\ast}$ and hence $\Ccal' = [f(V)]$ is the unique non-uniclass in $\gF(\Ccal)$ as claimed. 
\end{remark} 

In what follows, we fix the choice of $U_x$ for each $x \in \mathcal{O}_{f}(\Sing(f))$ to satisfy conditions in Lemma~\ref{lem:T} and the following property is satisfied: if $U \in \mathcal{U}_{p_{\gamma}}^{\ast}$ and $f(U)$ is disjoint from $U_{f(p_{\gamma})}$, then for any disk $U' \subseteq f(U)$ with $U' \in \mathcal{U}_{\ast}$, the $T$-class $[U']$ is a uniclass.

In the following subsections, we use the covering $\mathcal{U}_{\ast}$ of $\Jcal(f)$ and the $T$-classes $\mathcal{T}$ to encode the dynamical system $(\Jcal(f), f)$ in terms of symbolic dynamics-both as a system with (at most) countably many states and as a subshift of finite type.

\subsection{Symbolic dynamical systems}\label{sec:splitting}
In this subsection, we continue to assume that $\mathcal{J}(f)$ is compact. 
 Inspired by \cite{Fan21}, using the covering $\Ucal_{\ast}$ in previous subsection, we construct a symbolic dynamical system of at most countably many states and establish a topological conjugacy result under the assumption that $\mathrm{\bf Sing}(f)$ is normal.

\begin{proposition}\label{prop:countable}
Suppose that each point of $\Sing(f)$ is normal if it is nonempty.  Then, there exist  a symbolic dynamical system $(\Sigma_A,\sigma)$ of  at most countably many states corresponding to the disks in $\mathcal{U}_\ast$, and a bijection $h:\mathcal{J}(f)\to\Sigma_A$ such that $h$ is a topological conjugation in $\mathcal{J}(f)\setminus \mathrm{GO}\left(\Sing(f) \right)$ between $f$ and $\sigma.$ 
\end{proposition}

\begin{remark}\label{rem:hyperbolic}
In the case where $f$ is hyperbolic, since the Julia set $\Jcal(f)$ is compact and $\Sing(f)$ is empty,  Lemma~\ref{lem:markov} actually shows that the covering  $\mathcal{U}_\ast$ of $\Jcal(f)$ consists of finitely many scaling disks for $f$. Thus, the symbolic dynamics  $(\Sigma_A,\sigma)$  is a subshift of finite type. 
\end{remark}

\begin{proof}[Proof of Proposition \ref{prop:countable}]

Since every point in $\Sing(f)$ is normal when $\Sing(f)$ is nonempty, Lemma~\ref{lem:markov} implies that the Julia set $\Jcal(f)$ admits a covering $\Ucal_{\ast}$ satisfying the Markov property. We fix such a covering, which also refines the ones constructed in Lemma~\ref{lem:singular}. 
 
We regard elements in $\mathcal{U}_\ast$ as symbols. To avoid confusion, we let $\Scal_{A}$ be a set of symbols corresponding to elements in $\Ucal_{\ast}$ and let  $U_{\a}\in \Ucal_{\ast}$ be the scaling disk corresponding to  $\a\in \Scal_{A}$. Let $A$ be the adjacency matrix induced from the action of $f$ on $\mathcal{U}_\ast$.  Namely, we have $A_{ij} = 1$ if and only if $f(U_{i}) \supseteq U_{j}$ for $i, j \in \Scal_{A}.$  We let  $(\Sigma_A,\sigma)$ denote the  symbolic dynamical system induced by the adjacency matrix $A$ and define the map 
$$h:\mathcal{J}(f)\to\Sigma_A$$
by sending $x\in\mathcal{J}(f)$ to $h(x) = (\alpha_0,\alpha_1,\alpha_2,\dots)\in\Sigma_A$, where $\alpha_i$ corresponds to  the scaling disk in $\mathcal{U}_\ast$ containing $f^{\circ i}(x)$ for each $i\ge 0$. 

Note that if $\Sing(f)$ is empty, then by Remark~\ref{rem:hyperbolic}, the symbol set $\Scal_{A}$ is finite. In this case, the adjacency matrix $A$ is a finite square matrix, and the symbolic dynamical system $(\Sigma_A, \sigma)$ forms a subshift of finite type. In the following, we assume that the singular set $\Sing(f)$ is nonempty, and observe that there are finitely many elements $\alpha \in \Scal_{A}$ correspond to points in $\Ocal_{f}(\Sing(f))$, called a {\em singleton}, which will play a special role in the analysis below.

We first show that $h$ is surjective. Let $\boldsymbol{\alpha}:=(\alpha_0,\alpha_1,\dots)\in \Sigma_A$ be given. Suppose that there is some $j\ge 0$ such that $\alpha_j $  is a singleton. 
Let $j_{0}$ be the least of such  integers. If $j_{0} = 0$ then it's clear that $h(x) = \boldsymbol{\a}$ for $x\in \Ocal(\Sing(f))$  corresponding to $\a_{0}.$ If $j_{0} \ge 1$, then the scaling property of $f$ in the disks corresponding to $\alpha_i$ for  $0\le i\le j_{0}-1$ would imply that there exists $x\in\mathcal{J}(f)$ such that $h(x)=\boldsymbol{\alpha}$.  

We now assume that each $\alpha_i$ corresponding to a nontrivial disk  $U_i\in\mathcal{U}_\ast,$ where we denote  $U_{i} = U_{\a_{i}}$ for $i\ge 0$.  For each  $i\ge 1,$ define 
$$\Jcal_{i} := f^{-i} (U_{i}) \cap \Jcal(f).$$ 
 Note that $f(U_{i})\supseteq U_{i+1}$ for all $i\ge 0$. This implies that the sequence $\{\Jcal_{i}\}_{i\ge 1}$ forms a nested decreasing sequence of nonempty subsets of $\Jcal(f).$ Since the Julia set $\Jcal(f)$ is compact, the intersection 
 $$\bigcap_{i\ge 1} \, \Jcal_{i}$$ 
 is a nonempty  subset of $\Jcal(f).$ Then, we have that $h(x) = \boldsymbol{\a}$ for any $x\in \cap_{i\ge 1} \, \Jcal_{i}$ as desired.
%

To prove that $h$ is injective, suppose that $h(x)=h(y) =\boldsymbol{\a}=(\a_{0}, \a_{1}, \ldots )$ for some  $x, y\in\mathcal{J}(f)$. If both $x$ and $y$ lie in $ \GO(\Sing(f))$, then there exists a $j_{0} \ge 0$ such that $\a_{j}$ is a singleton for all $j\ge j_{0}$.  Arguing as in the previous paragraph for the case where $\a_{j}$  is a singleton for some $j$, one can deduce that $x = y$.  

Now suppose that each disk $U_{i}$ corresponding to $\a_{i}$ is a nontrivial scaling disk, and assume for contradiction  that $x\ne y$. Let  $\overline D$  be the smallest closed disk in $X$ containing both $x$ and $y$  so that $\diam(\overline{D}) = \rho(x, y)$. Since both $x$ and $y$ are in $\Jcal(f)$,  we have $\overline D\cap\mathcal{J}(f)\not=\emptyset$. By considering a suitable forward orbit of $\overline D$ under $f$, we may assume without loss of generality that $\overline D$ contains a repelling periodic point of period $s\ge 1$. Given the assumption  $h(x)=h(y)$, it follows  that $f^{\circ i}(\overline D) \subseteq U_{i}\in \Ucal_{\ast}$ for all $i\ge 0$. Therefore, $f^{\circ i}$ is scaling on $\overline{D}$ and in particular, we have 
 \[ \overline D\subsetneq f^{\circ s}(\overline D), \quad \text{with} \; \; \diam(f^{\circ (\ell s)}(\overline D)) = \lambda^{\ell} \diam(\overline D) 
 \]
for some constant $\l > 1$ and all $\ell \ge 0.$   
It follows that  there exists some $i_0\ge 1$ such that $f^{\circ(i_0s)}(\overline D)$ properly contains the disk $U_{0}$ (which already contains $\overline D$). That is, 
\[
  \overline{D} \subseteq U_{0} \subsetneq f^{\circ(i_0s)}(\overline D) \subseteq U_{i_{0} s} 
\]
This  contradicts the fact that $U_{0}$ and $U_{i_{0} s}$ are disjoint scaling disks in $\Ucal_{\ast}$. Therefore, our assumption $x\ne y$ must be false and we conclude that $x=y$. Hence, the map $h$ is injective.

Observe that $h\circ f=\sigma\circ h$ in $\mathcal{J}(f)$ by the construction of $h$. Indeed, for $x\in\mathcal{J}(f)$, let $h(x)=(\alpha_0,\alpha_1,\alpha_2,\dots)\in\Sigma_A$, then 
$$h\circ f(x)=(\alpha_1,\alpha_2,\dots)=\sigma\circ h(x).$$

Thus, to complete the proof, it remains to show that $h$  is continuous in $\mathcal{J}(f)\setminus\mathrm{GO}\left(\Sing(f)\right)$ and $h^{-1}$ is  continuous in $\Sigma_A\setminus h\left(\GO(\Sing(f))\right)$. Let $x\in \mathcal{J}(f)\setminus\GO(\Sing(f))$ and write $h(x)=(\alpha_0,\alpha_1,\dots,\alpha_{i_0}\dots)$. Note that an element $x$ is in $\mathcal{J}(f)\setminus\GO(\Sing(f))$ if and only if  the disks $U_{i}$ corresponding to the symbols $\a_{i}$ are  non-trivial scaling disks  for all $i\ge 0.$ For a cylinder set containing $h(x)$ of the following form
$$\mathcal{C}:=\{\boldsymbol{\beta}=(\beta_0,\beta_1,\dots):\beta_i=\alpha_i: 0\le i\le i_0\}\subseteq\Sigma_A,$$
 we have that 
$$h^{-1}(\mathcal{C})=\left(\bigcap_{i=0}^{i_0}f^{-i}(U_{\a_{i}})\right)\bigcap \Jcal(f),$$
Thus $h^{-1}(\mathcal{C})$ is open in $\Jcal(f)$ and hence $h$ is continuous. Conversely, for an open set $\Omega\subseteq\mathcal{J}(f)\setminus\GO(\Sing(f)$, shrinking $\Omega$ if necessary, we can assume that $\Omega =\Jcal(f) \cap U $ for some nontrivial disk $U$ in $\mathcal{U}_\ast$. Let $\ell\ge 1$ be the smallest integer such that $f^{\circ \ell}(U)$ contains at least two disks in $\mathcal{U}_\ast$, and let $U_i\in\mathcal{U}_\ast$ be the disk containing $f^{\circ i}(\Omega)$ for  $0\le i\le \ell-1$. 
Let  $u_i\in \Scal_{\ast}$ be the symbol corresponding to $U_i$ for each $0\le i\le \ell-1$,   
it follows from he definition of $h$ that 
$$h(\Omega)=\{\boldsymbol{\omega}:=(\omega_0,\omega_1,\dots):\omega_i=u_i\ \text{for}\ 0\le i\le\ell-1 \}$$
which is a cylinder set in $\Sigma_{A}$. Hence $h(\Omega)$ is an open set in $\Sigma_A$.  It follows that $h^{-1}$ is continuous as desired.  
\end{proof}

\subsection{Factorization and finite symbols}\label{sec:factoring}

In this subsection, we assume that the map $f$ meets the hypotheses of Theorem~\ref{thm:zeta}. Hence, the singular set $\Sing(f)$ satisfies both the Normal Condition and the Homogeneous Condition if it is nonempty. 
Under this assumption,  applying the Homogeneous Condition, we refine the symbolic dynamical system constructed in Proposition~\ref{prop:countable} to obtain a symbolic dynamical system that is a subshift of finite type.

\begin{proposition}\label{prop:finite}
 There exist a symbolic dynamical system $(\Sigma_B,\sigma)$ which is a subshift of finite type  and an injection $g:\mathcal{J}(f)\to\Sigma_B$ such that the following hold:
\begin{enumerate}
\item $g$ is a topological conjugation on $\mathcal{J}(f)\setminus \GO(\Sing(f))$ between $f$ and $\sigma$.
\item The complement $\Sigma_B\setminus g(\mathcal{J}(f))$ consists of the grand orbits of finitely many periodic cycles in $(\Sigma_B,\sigma)$. 
\end{enumerate}
\end{proposition}

To prove Proposition~\ref{prop:finite}, we proceed through a sequence of lemmas presented below. The argument is organized into three main parts. In \S\ref{sec:equivalence}, we define a factor map $\pi$ from the symbolic space $\Sigma_A$ (defined in Proposition~\ref{prop:countable}) to $\Sigma_B$. 
 In \S\ref{sec:truncation}, we introduce the notion of $\g$-shadows to describe the structure of elements in both $\Sigma_B$ and the image $\pi(\Sigma_A)$. Finally, in \S\ref{sec:pr-finite}, we bring these tools together to complete the proof of Proposition~\ref{prop:finite}.

Throughout the argument, we assume that for each $x \in \Ocal_{f}(\Sing(f))$, the closed disk $U_x$ is chosen as in \S\ref{subsec:homo}. This ensures that the $T$-equivalence relation is well-defined on the covering $\Ucal_{\ast}$, and that the map $\mathfrak{F}$ on $\Tcal$ is valid.

\subsubsection{Factor map}\label{sec:equivalence} 

Let $\Scal_{B}$ denote the set of symbols corresponding to elements in $\Tcal.$ Then, the map $\mathfrak{F}$ induces a natural adjacency matrix $B$ on $\Scal_{B}.$ Namely,  for $i, j \in \Scal_{B}$, the $(i,j)$-entry $B_{ij} $ of $B$ is $1$ if $\Ccal_{j} \in \gF(\Ccal_{i})$ and is $0$ otherwise ,  where $\Ccal_{i}$ denotes the $T$-class corresponding to the symbol $i\in \Scal_{B}$,  Note that $B$ is a finite square matrix since $\Tcal$ is finite by Lemma~\ref{lem:finite T-class}. 

Let $(\Sigma_{B}, \s)$ be the symbolic dynamical system induced by the adjacency matrix $B$. Then, $(\Sigma_{B}, \s)$ is a subshift of finite type. We will show that the dynamical system $(\Jcal(f), f)$ can be encoded in 
$(\Sigma_{B}, \s)$.  Let us first make the following definition. 

\begin{definition}
For two symbolic dynamical systems $(\Sigma_{A_1},\sigma)$ and $(\Sigma_{A_2},\sigma)$ over symbol sets $\mathcal{A}_1$ and $\mathcal{A}_2$, we say a map $\pi:\Sigma_{A_1}\to\Sigma_{A_2}$ is a \emph{factor map} if the following hold:
\begin{enumerate}
\item $\pi$ is continuous,
\item there exists a surjection $\theta: \mathcal{A}_1\to\mathcal{A}_2$ such that $\pi((\alpha_1,\alpha_2,\dots))=(\theta(\alpha_1),\theta(\alpha_2),\dots)$ for any $(\alpha_1,\alpha_2,\dots)\in\Sigma_{A_1}$. 
\end{enumerate}
\end{definition}

Observe that any factor map $\pi:\Sigma_{A_1}\to\Sigma_{A_2}$ satisfies that $\sigma\circ\pi=\pi\circ\sigma$ since for any $(\alpha_1,\alpha_2,\dots)\in\Sigma_{A_1}$, 
$$\sigma\circ\pi((\alpha_1,\alpha_2,\dots))=(\theta(\alpha_2),\dots)=\pi\circ\sigma((\alpha_1,\alpha_2,\dots)).$$

\noindent We now introduce a factor map from the symbol space $\Sigma_A$ to the symbol space $\Sigma_B$. Define the map $\theta: \Scal_{A} \to \Scal_{B}$ by sending $\a \in \Scal_{A}$ to the symbol corresponding to the $T$-class $[U_{\a}] \in \mathcal{T}$. By construction, $\theta$ is a surjective map. Using this, we define  
$$\pi: \Sigma_A\to\Sigma_B,$$
by setting 
$$\pi\left((\alpha_0,\alpha_1,\dots)\right) := (\theta(\alpha_0),\theta(\alpha_1),\dots)$$
for any $(\alpha_0,\alpha_1,\dots) \in \Sigma_{A}.$ 

\begin{lemma}\label{lem:finite}
The map $\pi:\Sigma_A\to\Sigma_B$ is a factor map. 
\end{lemma}
\begin{proof}
By the definition of $\pi$, it suffices to show that $\pi$ is continuous. Let $\boldsymbol{\a} = (\a_0, \a_1, \dots) \in \Sigma_A$ be given, and let $(\boldsymbol{\a}^{(m)})_{m \ge 1} := \left((\a^{(m)}_0, \a^{(m)}_1, \dots)\right)_{m \ge 1}$ be a sequence in $\Sigma_A$ such that $\boldsymbol{\a}^{(m)} \to \boldsymbol{\a}$ as $m \to \infty$. This means that for each $j \ge 0$, there exists an integer $m_j \ge 1$ such that $\a_i^{(m)} = \a_i$ for all $0 \le i \le j$ and all $m \ge m_j$.

Since $\theta(\a_i^{(m)}) = \theta(\a_i)$ for all $0 \le i \le j$ and $m \ge m_j$, it follows that $\pi(\boldsymbol{\a}^{(m)}) \to \pi(\boldsymbol{\a})$ as $m \to \infty$. Therefore, $\pi$ is continuous, as claimed.
\end{proof}

\subsubsection{$\g$-Shadow}\label{sec:truncation}
We now study the elements in the image $\pi(\Sigma_A)$ of the factor map $\pi$. To better understand the structure of $\pi(\Sigma_A) \subseteq \Sigma_B$, we consider a (finite or infinite) admissible sequence  $(\beta_1, \beta_2, \dots, \beta_r)\in \Scal_{B}^{r}, \; r\in \NN\cup\{\infty\}$ with respect to the adjacency matrix $B$. 

We are concerned with whether or not for such an admissible sequence  $(\beta_1, \beta_2, \dots, \beta_r)$ in $\Scal_B^{r}$, there exists a corresponding admissible sequence  $(\alpha_1, \alpha_2, \dots, \alpha_r) \in \Scal_A^{r}$ such that 
$$(\beta_{1},\beta_{2},\dots,\beta_{r}) = (\theta(\a_{1}), \theta(\a_{2}),\dots,\theta({\a_{r}})).$$ 


To address this question, we introduce the notion of {\em $\g$-shadow} as follows. For each $x\in\Ocal_{f}(\Sing(f))$, recall that $U_x\subseteq\overline{D}_x$ denotes the domain of the map $T_x$ and let $\Ucal_{x}^{\ast}$ be the sub-collection of scaling disks in $U_{x}$ as defined in \eqref{eqn:Ucalx}.  For the purposes of our subsequent discussion, we define the collection of scaling disks associated to a repelling cycle $\gamma \in \Gamma(f)$ as 
$$
\Ucal_{\g}^{\ast} := \bigcup_{x\in \GO^{\ast}(\g)}\, \Ucal_{x}^{\ast}. 
$$
\begin{definition}\label{def:gamma shadow} 
Let  $\boldsymbol{\beta}=(\beta_{0},\beta_{1},\dots)\in\Sigma_B$  with each  $\beta_i$ corresponds to the $T$-class $\Ccal_{i} \in\Tcal$. We say that a subsequence $(\beta_{j},\beta_{j+1},\dots,\beta_{j+s})$ of $\boldsymbol{\beta}$, where $j\ge 0$ and $0\le s\le+\infty$, is  a \emph{$\g$-shadow} of $\boldsymbol{\beta}$ if there exists a repelling cycle $\g\in\Gamma(f)$ such that the following conditions hold:
\begin{enumerate}
\item for each  $j\le i\le j+s$,  $\Ccal_{{i}}\in \Ccal_{T}(\Ucal_{x_i}^{\ast})$ for some $x_i\in\mathrm{GO}^*(\gamma)$, 
\item if $j\ge 1$, then $\Ccal_{{j-1}}\notin \Ccal_{T}(\Ucal_{\g}^{\ast}) :=\cup_{x\in \GO^{\ast}(\g)} \Ccal_{T}(\Ucal_{x}^{\ast})$, and  
\item if $s\not=+\infty$, then $\Ccal_{{j+s+1}}\notin \Ccal_{T}(\Ucal_{\g}^{\ast})$ 
\end{enumerate}
Moreover,  we say that the sequence  $(\beta_{j},\beta_{j+1},\dots,\beta_{j+s})$ is an \emph{infinite $\g$-shadow} if $s=+\infty$,  and  a \emph{finite $\g$-shadow}, otherwise. 
\end{definition}

Note that if a subsequence   $(\beta_{j},\beta_{j+1},\dots,\beta_{j+s})$ of $\boldsymbol{\beta}$ is a finite $\g$-shadow, then by Remark~\ref{rem:acting on T-class}, the symbol $\beta_{j+s+1}$ corresponds to a uniclass of $\Tcal$. Moreover, if $j\ge 1$, then for the same reason, $\beta_{j-1}$ also corresponds to a uniclass in $\mathcal{T}$.

\begin{lemma}\label{lem:finite-T}
Let $(\beta_{j},\beta_{j+1},\dots,\beta_{j+s})$ be a finite $\g$-shadow of $\boldsymbol{\beta}=(\beta_{0},\beta_{1},\dots)\in\Sigma_B$. Then, there is a unique admissible sequence $(\a_{j},\a_{j+1},\dots,\a_{j+s})$ in $\Scal_{A}^{s+1}$ 
 such that 
 $$(\beta_{j},\beta_{j+1},\dots,\beta_{j+s}) = (\theta(\a_{j}), \theta(\a_{j+1}),\dots, \theta(\a_{j+s})).$$ 
\end{lemma}
\begin{proof}

By definition, each symbol $\beta_{i}$ corresponds to a $T$-class $\mathcal{C}_{i} \in \mathcal{C}_{T}(\mathcal{U}_{x_{i}}^{\ast})$ for some $x_{i} \in \mathrm{GO}^{\ast}(\gamma)$, where $j \le i \le j+s$ and $x_{i+1} = f(x_i)$ for $j \le i \le j+s-1$. We claim that there exists a unique scaling disk $U_i \in \mathcal{C}_i$ such that
 $$U_{i+1} \subset f(U_{i})\quad   \text{for each $j \le i \le j+s.$ } $$

We begin with the case where $i = j + s$. By the remark above, the symbol $\beta_{j+s+1}$ corresponds to a uniclass $[U_{j+s+1}]$ consisting of a single scaling disk $U_{j+s+1} \in \mathcal{U}_{\ast}$.  Then, by definition of admissibility of the sequence $\boldsymbol{\beta}$ with respect to $B$, we have $[U_{j+s+1}] \in \mathfrak{F}(\mathcal{C}_{j+s})$. Therefore, there exists some $U \in \mathcal{C}_{j+s}$ such that $U_{j+s+1} \subset f(U)$. 

According to Remark~\ref{rem:acting on T-class}, we conclude that $\mathcal{C}_{j+s} \in \mathcal{C}_{T}(\mathcal{U}_{p}^{\ast})$ and that $U = U_{{\mathcal{C}_{j+s}}}$ is the terminal member of $\mathcal{C}_{j+s}$, where $p = p_{\gamma}$ is the marked point for the repelling cycle $\gamma$. This proves the claim for the case where $i=j+s.$

Now, consider the case where $i < j+s$ and assume that $\Ccal_{\ell} = [U_{\ell}]$ with $U_{\ell+1} \subset f(U_{\ell})$ for $i+1\le \ell \le j+s.$ Since both $x_{i}$ and $x_{i+1}= f(x_{i})$ are in $\GO^{\ast}(\g)$, as discussed in Remark~\ref{rem:acting on T-class}, $\Ccal_{i+1} = [U_{i+1}]$ is the unique non-uniclass in $\gF(\Ccal_{i})$ and $U_{i+1} = f(U_{i})$ for a unique scaling disk $U_{i} \in \Ccal_{i}$.  Inductively, we can conclude that for each $j\le i\le j+s$, we have unique $U_{i} \in \Ucal_{x_{i}}$ satisfying  $\Ccal_{i} = [U_{i}]$ and $U_{i+1} \subset f(U_{i})$. 

\end{proof}

\begin{corollary}\label{coro:finite}
Let  $\boldsymbol{\beta}\in\Sigma_B$ such that it has no infinite $\g$-shadow. Then $\boldsymbol{\beta}\in\pi(\Sigma_A)$.
\end{corollary}
\begin{proof}
Let $\boldsymbol{\beta} = (\beta_{0}, \beta_{1}, \ldots)$ be an element of $\Sigma_{B}$ that does not contain an infinite $\gamma$-shadow. If $\beta_{i}$ corresponds to a uniclass $[U_{i}]$ for some $U_{i} \in \mathcal{U}_{\ast}$, then clearly $\beta_{i} = \theta(\alpha_{i})$, where $\alpha_{i} \in \mathcal{S}_{A}$ is the symbol corresponding to $U_{i}$.

Now, suppose that $(\beta_{j}, \ldots, \beta_{j+s})$ is a finite $\gamma$-shadow of $\boldsymbol{\beta}$. By Lemma~\ref{lem:finite-T}, there exists a block $(\alpha_{j}, \ldots, \alpha_{j+s+1}) \in \mathcal{S}_{A}^{s+2}$ such that 
$$(\beta_{j}, \ldots, \beta_{j+s+1}) = (\th(\a_{j}), \ldots, \th(\a_{j+s+1}))$$
It remains to show that, when $j \ge 1$, the extended block $(\beta{j-1}, \ldots, \beta_{j+s+1})$ is also in the image of the factor map $\pi$.

Note that in this case, $\beta_{j-1}$ also corresponds to a uniclass $\mathcal{C}_{j-1} = [U_{j-1}]$. By definition of the map $\mathfrak{F}$, we must have $\mathcal{C}_j \in \mathfrak{F}(\mathcal{C}_{j-1})$. Therefore, there exists a scaling disk $U_{j}' \in \mathcal{C}_j$ such that $U_{j}' \subset f(U_{j-1})$

Then, by the uniqueness part of Lemma~\ref{lem:finite-T}, we conclude that $U_{j}' = U_{j}$. Hence, the block $(\alpha_{j-1}, \ldots, \alpha_{j+s+1})$ can be lifted to $\Scal_A^{s+3}$ under $\pi$.

In conclusion, we have shown that all such finite blocks containing a $\g$-shadow, and hence the entire sequence $\boldsymbol{\beta}$, lie in the image of the factor map $\pi$.
\end{proof}

We now  consider the elements in $\Sigma_B$ having infinite $\g$-shadows.  Recall the integer valued function $\varphi: \Gamma(f) \to \ZZ$ from Definition \ref{def:homogeneous}.

\begin{lemma}\label{lem:in}
Let $\boldsymbol{\beta} = (\beta_{0}, \beta_{1}, \ldots)\in\Sigma_B$ be a sequence with infinite $\g$-shadow for some $\g\in \Gamma(f)$. Then, $\boldsymbol{\beta}$ is eventually mapped, under $\sigma$, to a periodic cycle in $\Sigma_B$ whose length is a factor of $\varphi(\gamma)$. 

\end{lemma}
\begin{proof}

Let $\boldsymbol{\beta}\in \Sigma_{B}$ be a given sequence with infinite $\g$-shadow. By definition,  there exists an integer $j_{0}\ge 0$ such that  for all $i \ge j_{0},$ each $\beta_{i}$ corresponds to a $T$-class $\Ccal_{i} \in \Ccal_{T}(\Ucal_{\g}^{\ast}).$

If for some $i\ge j_{0}$, $\beta_{i}$ corresponds to a uniclass in $\Ucal_{x_{i}}^{\ast}$ (where $x_{i} \in  \GO^{\ast}(\g)$),  then this uniclass  must be the unique class containing the singleton  $\{x_{i}\}$ in $\Ucal_{x_{i}}^{\ast}$. In this case, for all $i' \ge i$, the entry  $\beta_{i'}$ must correspond to the uniclass $[\{x_{i'}\}]$, where $x_{i'} \in \Ocal_{f}(x_{i})$.  Consequently, the sequence  $\boldsymbol{\beta}$ is mapped to a periodic cycle with entries consist of uniclasses $[\{x\}]$ with $x\in \g$, whose length  $s_{\g}$ divides $\varphi(\g).$

Suppose that for all sufficiently large $i$, each entry  $\beta_{i}$  of $\boldsymbol{\beta}$ corresponds to a non-uniclass $\Ccal_{i}\in \Ccal_{T}(\Ucal_{\g}^{\ast}).$  Then, there exists  a $j_{0}\ge 0$ such that  $\Ccal_{j_{0}}\in \Ccal_{T}(\Ucal_{p}^{\ast}).$ Since $\Ccal_{j_{0}+1}$ is also a non-uniclass, by  Remark~\ref{rem:acting on T-class}, it must be the unique non-uniclass in $\gF(\Ccal_{j_{0}})$.   

As a result, the subsequence $\boldsymbol{\beta'}= (\beta_{j_{0}}, \beta_{j_{0}+1}, \ldots )$ must satisfy $\beta_{i + \varphi(\g)} = \beta_{i}$ for all $i\ge j_{0}.$ Thus, $\boldsymbol{\beta'}$ forms a cycle in $\Sigma_{B}$ whose length divides $\varphi(\g)$,  completing the proof. 
\end{proof}

\begin{corollary}\label{coro:infinite}
There are only finitely many periodic elements in $\Sigma_B$ having infinite $\g$-shadows.
\end{corollary}
\begin{proof}
Since $\mathcal{T}$ is finite, given an integer $s\ge 1$, there are only finitely many periodic cycles in $\Sigma_B$ of length $s$. Thus the conclusion follows from Lemma \ref{lem:in}. 
\end{proof}

We now establish a correspondence between the points in $\Sigma_A$ whose images  under the factor map $\pi:\Sigma_A\to\Sigma_B$  contain an infinite  $\g$-shadow, and   the points in $\mathcal{J}(f)$ that are eventually mapped to a repelling cycle  in $\Gamma(f)$. 
\begin{lemma}\label{lem:infinite-T}
Let $\boldsymbol{\alpha}\in\Sigma_A$ such that $\pi(\boldsymbol{\alpha})$ contains an infinite $\g$-shadow. Then 
$$h^{-1}(\boldsymbol{\alpha})\in\mathrm{GO}(\Sing(f))$$
where $h:\mathcal{J}(f)\to\Sigma_A$ is the map defined in Proposition~\ref{prop:countable}.  
\end{lemma}
\begin{proof}
Let $\boldsymbol{\alpha}= (\a_{0}, \a_{1}, \ldots )$ be a given sequence in $\Sigma_{A}$ such that 
$\pi(\boldsymbol{\alpha})$ contains an infinite $\g$-shadow. Let $(\th(\a_{j_{0}}), \th(\a_{j_{0}+1}), \ldots)$ be an infinite   $\g$-shadow of $\pi(\boldsymbol{\alpha}).$ It suffices to show that there exists a $j\ge j_{0}$ such that $\a_{j}$ corresponds to a singleton $\{x\}\in \Ucal_{x}$ for some $x\in \g.$ Equivalently, $\th(\a_{j})$ corresponds to the uniclass $[\{x\}]$ in $\Ccal_{T}(\Ucal_{x})$

For the purpose of contradiction, suppose that for all $j\ge j_{0}$, each symbol $\th(\a_{j})$ represents a non-uniclass.  Then, there exists some $j \ge j_{0}$ such that $\th(\a_{j})$ corresponds to a non-uniclass $\Ccal_{j} \in \Ccal_{T}(\Ucal_{p}).$ By Lemma~\ref{lem:in}, we have  
$$\th(\a_{j}) = \th(\a_{j+\ell \varphi(\g)})\quad \text{for all integer $\ell\ge 0.$}
$$
Let $U_{j}\in \Ucal_{p}^{\ast}$ be the scaling disk associated with $\a_{j}$. Then, $\a_{j+\ell \varphi(\g)}$ is the symbol corresponds to the scaling disk $f^{\circ (\ell \varphi(\g))}(U_{j})$, and $\th(\a_{j+\ell \varphi(\g)})$ corresponds to the $T$-class $[f^{\circ (\ell \varphi(\g))}(U_{j})]=\Ccal_{j}.$ 

However, we can choose $\ell$  sufficiently large such that $f^{\circ (\ell \varphi(\g))}(U_{j})$ is disjoint from $U_{p}$. By Lemma~\ref{lem:T-equiv} (see also Remark~\ref{rem:acting on T-class}), the $T$-class $[f^{\circ (\ell \varphi(\g))}(U_{j})]$ associated with $\th(\a_{j+\ell \varphi(\g)}) $ is a uniclass, contradicting initial assumption that  all  $\th(\a_{i})$ were non-uniclass for $i\ge j_{0}$,  thereby completing the proof.   
\end{proof}

\subsubsection{Proof of Proposition \ref{prop:finite}}\label{sec:pr-finite}
To establish Proposition \ref{prop:finite}, we first show the injectivity of the factor map $\pi:\Sigma_A\to\Sigma_B$. 

\begin{lemma}\label{lem:factor-inj}
The factor map $\pi:\Sigma_A\to\Sigma_B$ is injective. 
\end{lemma}
\begin{proof}
Let $\boldsymbol{\beta} = (\th(\a_{0}), \th(\a_{1}), \ldots) \in \pi(\Sigma_{A})$ be given. Note that if a symbol $\beta \in \Scal_{B}$ corresponds to  a  uniclass $[U]$, then clearly $\beta = \th(\a)$, where $\a$ is the (unique) symbol corresponding to $U.$ Therefore, we only need to consider the case where $\th(\a_{i})$ is not a uniclass -- that is, $\boldsymbol{\beta}$ has a $\g$-shadow for some $\g\in \Gamma(f).$ 
 
If $\boldsymbol{\beta}$ has an infinite $\g$-shadow $(\beta_{j_{0}}, \beta_{j_{0}+1}, \ldots)$, then the proof of   Lemma~\ref{lem:infinite-T} actually shows that each symbol  $\beta_{j}$ corresponds to the uniclass $[\{x_{j}\}]$ where $x_{j} \in \GO^{\ast}(\g)$ for all $j\ge j_{0}.$ 

If $\boldsymbol{\beta}$ instead has a finite $\g$-shadow  $(\beta_{j_{0}}, \ldots, \beta_{j_{0}+s})$ for some $s\ge 1,$ then by Lemma~\ref{lem:finite-T}, there exists a unique admissible sequence $(\a_{j_{0}}, \ldots, \a_{j_{0}+s})\in \Scal_{A}^{s+1}$ such that 
$$(\beta_{j_{0}}, \ldots, \beta_{j_{0}+s}) = (\th(\a_{j_{0}}), \ldots, \th(\a_{j_{0}+s}).$$ 
In conclusion, we have shown in all cases that there exists a unique sequenc  $\boldsymbol{\a} = (\a_{0}, \a_{1}, \ldots)$ in $\Sigma_{A}$ such that $\boldsymbol{\beta} =\pi(\boldsymbol{\a}).$ 
\end{proof}

\noindent Now we are ready to  prove Proposition~\ref{prop:finite}. 
\begin{proof}[Proof of Proposition \ref{prop:finite}]
Let $h$ be as in Proposition  \ref{prop:countable}  and consider the factor map $\pi:\Sigma_A\to\Sigma_B$ as in Lemma \ref{lem:finite}. Define $g=\pi\circ h$. Since the two maps $\pi$ and $h$ are injective by Proposition  \ref{prop:countable} and Lemma \ref{lem:factor-inj}, we conclude  that $g$ is injective. 

Statement (1) follows immediately from Proposition  \ref{prop:countable}, Lemmas \ref{lem:finite} and \ref{lem:factor-inj}. Statement (2) is a consequence of  Corollary~\ref{coro:finite}. Indeed, since every element of  $\Sigma_B\setminus g(\mathcal{J}(f))$ has an infinite $\g$-shadow by  Corollary~\ref{coro:finite}, it follows from Lemmas~\ref{lem:in} and \ref{lem:infinite-T}  that $\Sigma_B\setminus g(\mathcal{J}(f))$ is nonempty and that each such element is eventually mapped to a periodic cycle. Moreover, Corollary~\ref{coro:infinite} shows that there are only finitely many such periodic cycles. This completes the proof of Statement (2).
\end{proof}

\subsection{Proof of Theorem \ref{thm:zeta}}\label{sec:pr}
 Proposition \ref{prop:finite} implies  that for each $n\ge 1$, the set $\mathrm{Fix}^*(f^{\circ n})$ is finite, since $\sigma^{\circ n}$ has finitely many fixed points in $\Sigma_B$. Moreover, 
 it follows from Proposition \ref{prop:finite} (1) and Lemmas \ref{lem:in} and \ref{lem:infinite-T} that  the image $g(\gamma)$ of a repelling cycle $\gamma$ in $\mathcal{J}(f)$ of length $s\ge 1$ is a periodic cycle of $\sigma$ in $\Sigma_B$ of length $s$. 
 
 Let $\boldsymbol{\beta}_1,\dots,\boldsymbol{\beta}_{m_0}$ be all the periodic cycles in $\Sigma_B\setminus g(\mathcal{J}(f))$, and let $\ell_m$ be the length of $\boldsymbol{\beta}_m$ for each $1\le m\le m_0$. It follows from Corollaries \ref{coro:finite} and \ref{coro:infinite} that $m_0$ is finite. We conclude that 
 $$\zeta_f(X,t)= \frac{\zeta_\sigma(\Sigma_B,t)}{\exp\left(\sum_{m=1}^{m_0}\left(\sum_{j=1}^\infty\left(\ell_m\frac{t^{(j \ell_m)}}{j\,\ell_m}\right)\right)\right)},$$
 where
 $\zeta_\sigma(\Sigma_B,t)$ 
 is the Artin-Mazur zeta function associated with  $(\Sigma_B,\sigma)$. In view of \S\ref{sec:AM},  
 $$\zeta_\sigma(\Sigma_B,t)=\det(I-tB)^{-1}.$$
 Noting that
 $$\exp\left(\sum_{m=1}^{m_0}\left(\sum_{j=1}^\infty\left(\ell_m\frac{t^{(j\ell_m)}}{j\,\ell_m}\right)\right)\right)=\prod_{m=1}^{m_0}\frac{1}{1-t^{\ell_m}},$$
 we conclude that
$$\zeta_f(X,t)=\frac{\prod_{m=1}^{m_0}(1-t^{\ell_m})}{\det(I-tB)}.$$
Observe that both $\prod_{m=1}^{m_0}(1-t^{\ell_m})$ and $\det(I-tB)$ are monic polynomials with integer coefficients. Define 
$$G(t):=\gcd(\prod_{m=1}^{m_0}(1-t^{\ell_m}),\det(I-tB) )$$ 
which is a monic polynomial whose zeros  are roots of unity.
It follows that both 
$$\frac{\prod_{m=1}^{m_0}\,(1-t^{\ell_m})}{G(t)}\quad  \text{and} \quad \frac{\det(I-tB)}{G(t)}$$ 
are monic and relatively prime  polynomials  in $\mathbb{Z}[t]$. Therefore $\zeta_f(X,t)$ is of desired form. \qed 

 \section{Rational maps case}\label{sec:rationalmap}
 
This section is devoted to proving Theorems \ref{ratzeta} and \ref{main}.  We establish Theorem \ref{ratzeta}  in \S\ref{sec:proof-rational} and Theorem \ref{main}  in \S\ref{sec:entropy}, respectively. 
 
 \subsection{Proof of Theorem \ref{ratzeta}}\label{sec:proof-rational} 
 Before we proceed the proof, let us recall the settings from  Theorem~\ref{ratzeta} where $K$ is a complete non-archimedean field of characteristic $0$ and $\phi(z) \in K(z)$ is a subhyperbolic rational map with nonempty, compact $K$-rational Julia set $\Jcal(\phi, K)$ in $\PP^{1}(K).$ The set $\PP^{1}(K)$ of $K$-rational points of the projective line is naturally equipped with a non-archimedean metric induced by the absolute value of $K$ and $\phi$ is continuous self-map on $\PP^{1}(K).$  In view of \S\ref{sec:map} and Theorem \ref{thm:zeta},  to establish Theorem \ref{ratzeta}, it suffices to show that the critical points of $\phi$ satisfies the Normal Condition  and $\phi$ satisfies the Homogeneous Condition.

  The normality of critical points follows immediately from \cite[Theorem 9.46]{Baker10} and \cite[Theorem D]{Faber13}. It remains to show that  $\phi$ satisfies the  Homogeneous Condition. We  regard $\phi\in K(z)$ as a rational map in $\Kbar(z)$, where $\Kbar$ is  an algebraic closure of $K$, and denote by $m_x(\phi)$ the local degree of $\phi\in \Kbar(z)$ for any $x\in \mathbb{P}^1(\Kbar)$. Recall that $\Gamma(\phi)$ is the set of repelling cycles in $\mathcal{J}(\phi,K)$, each of which has  grand orbit  containing a critical point in $\mathrm{Crit}(\phi, K)$. 
 
 \begin{lemma}\label{lem:rational-homo}
Any subhyperbolic rational map $\phi\in K(z)$ of degree at least $2$ satisfies  the Homogeneous Condition. 
 \end{lemma}
 \begin{proof}
 For any $\gamma\in\Gamma(\phi)$ with a marked point $p_\gamma$, set 
 $$\delta_{\gamma}:=s_\gamma\cdot\prod_{c\in\mathrm{GO}(\gamma)\cap\mathrm{Crit}(\phi, K)}m_c(\phi),$$
 where $s_\gamma$ is the length of $\gamma$. 
 We define the map
 $$
\begin{array}{cccc}
\varphi:&\Gamma(\phi)&\to&\mathbb{Z}_{>0}\\
&\gamma&\to&\delta_\gamma.
\end{array}
$$

Pick $c\in\mathrm{GO}(\gamma)\cap\mathrm{Crit}(\phi, K)$ and let $\ell_c\ge 1$ be the smallest integer such that $\phi^{\circ\ell_c}(c)=p_\gamma$.  Conjugating $\phi$ by a M\"obius transformation over $K$, we can assume that $c\ne \infty$ (i.e. $c\in K$) and $p_{\g}=0$. 
Then in a small neighborhood $U_{p_{\g}}\subseteq K$ of $p_{\g}$, we can write 
$$\phi^{\circ\delta_\gamma}(z)=\lambda^{\delta_\gamma/s_\gamma}z+\sum_{i=2}^\infty a_iz^i,$$
where $\lambda=(\phi^{\circ s_\gamma})'(0)$ and $a_i\in K$ for each $i\ge 2$. 
Moreover, in a sufficiently small neighborhood $U_c\subseteq K$ of $c$ with $\phi^{\circ\ell_c}(U_c)\subseteq U_p$, we can write  
$$\phi^{\circ\ell_c}(z)=\sum_{i=0}^\infty b_i(z-c)^{m+i},$$
where $b_0\in K^{\ast} = K\setminus\{0\}$, $b_i\in K$ for each $i\ge 1$, and $m=m_c(\phi^{\circ\ell_c})$. 

Shrinking $U_c$ if necessary, we claim that, in $U_c$, there exists 
$$T_c(z)=c+\sum_{i=1}^\infty\alpha_i(z-c)^i$$
with $\alpha_1\in K^{\ast}$ and $\alpha_i\in K$ for each $i\ge 2$
 such that 
$$\phi^{\circ\ell_c}\circ T_c=\phi^{\circ(\delta_\gamma+\ell_c)}.$$
Observing that $m$ divides $\delta_\gamma/s_\gamma$, we can take  
$$\alpha_1=\lambda^{\delta_\gamma/(s_\gamma\cdot m)}\in K\setminus\{0\},$$
and inductively, 
 for each $i\ge 1$, comparing the coefficients of $(z-c)^{m+i}$ in $\phi^{\circ\ell_c}\circ T_c$ and in $\phi^{\delta_\gamma+\ell_c}$, we obtain that 
$$b_0\alpha_1^{m-1}\alpha_{i+1}=P(a_2,\dots,a_i,b_0,\dots,b_i,\alpha_1,\dots,\alpha_i), $$
where $P\in \mathbb{Z}[z_1,\dots,z_{3i}]$ is a polynomial, and conclude that $\alpha_{i+1}\in K$ is unique. Thus $T_c$, when exists, is unique under the choice of $\alpha_1$.     Moreover, since there exist analytic maps  $L_1(z)$ and  $L_2(z)$ defined over $\Kbar$ such that $L_1$ is a scaling near $ \phi^{\circ\ell_c}(c)$, $L_2$ is a scaling near $c$ and 
$$L_1\circ \phi^{\circ\ell_c}\circ L_2(z)=z^n$$
 for some integer $n\ge 2$ (see, e.g., \cite[\S 1]{Ingram13} and \cite[Theorem 3 and Corollary 6]{Salerno20}), it follows that 
$$(L^{-1}_2\circ T_c)^n=L_1\circ \phi^{\circ\ell_c}\circ L_2(z)\circ L^{-1}_2\circ T_c=L_1\circ\phi^{\circ(\delta_\gamma+\ell_c)},$$
and hence $T_c$ is analytic at $c$. Here the existence and analyticity of $T_c$ is due to the $n$-th roots of $L_1\circ\phi^{\circ(\delta_\gamma+\ell_c)}$ being analytic. Further shrinking  $U_c$ if necessary, we conclude that $T_c$ is a scaling in $U_c$.

To complete the proof, we need to consider the case where $\mathrm{GO}(\gamma)\cap\mathrm{Crit}(\phi, K)$ contains two distinct critical points, say $c_1$ and $c_2$, such that $\phi^{\circ j_0}(c_1)=c_2$ for some $j_0\ge 1.$ We must check $\phi^{\circ j_0}\circ T_{c_1}=T_{c_2}\circ\phi^{\circ j_0}$. 

By previous paragraph, we only need to show that there exists a scaling map $f$, defined on a neighbourhood of $c_1$ such that  $f(c_1)=c_{1}$ and    $f'(c_1):=\lambda^{\delta_\gamma/(s_\gamma\cdot m_1)} $, where $m_1=m_{c_1}(\phi^{\circ\ell_{c_1}})$. Furthermore, $f$ satisfies  $\phi^{\circ j_0}\circ f=T_{c_2}\circ\phi^{\circ j_0}$. Indeed, such a map $f$ satisfies $\phi^{\circ\ell_{c_1}}\circ f=\phi^{\circ (\delta_\gamma+\ell_{c_1})}$, which implies that $f=T_{c_1}$ by the uniqueness of $T_{c_1}$. Post-and pre-composing translations in $K[z]$, we can assume that  $c_1=c_2=0$ and that 
$$\phi^{j_0}(z)=\sum_{j\ge 0}B_jz^{d+j},$$
where $d=m_{c_1}(\phi^{j_0})$, $B_0\in K\setminus\{0\}$ and $B_j\in K$ for $j\ge 1$. From the previous paragraph, we have that 
$$T_{c_2}(z)=\lambda^{\delta_\gamma/(s_\gamma\cdot m_2)}z+\sum_{j\ge 2}\beta_jz^j,$$
where $m_2=m_{c_2}(\phi^{\circ\ell_{c_2}})$ and $\beta_j\in K$ for $j\ge 2$. We claim there exists a desired $f(z)=\sum_{j\ge 1}A_jz^j$. To determine $A_j$, we compare the coefficients of  $\phi^{\circ j_0}\circ f$ and $T_{c_2}\circ\phi^{\circ j_0}$. First, from the coefficients of $z^d$, we obtain 
$$A_1^d=\lambda^{\delta_\gamma/(s_\gamma\cdot m_2)}.$$
Since $d\cdot m_2=m_1$ and $m_1$ divides $\delta_\gamma/s_\gamma$, we can let 
$$A_1=\lambda^{\delta_\gamma/(s_\gamma\cdot m_1)}.$$
For each $j\ge 1$, comparing the coefficients of $z^{d+j}$, we obtain that 
$$B_0A_1^{d-1}A_{j+1}=
Q(A_1,\dots,A_j,B_0,\dots,B_j,\beta_2,\dots,\beta_j), $$
where $Q\in \mathbb{Z}[z_1,\dots,z_{3j}]$ is a polynomial, which implies that $A_{j+1}\in K$ exists. This shows the existence of $f$ and finishes the proof. 
\end{proof}

 Lemma \ref{lem:rational-homo} completes the proof of Theorem \ref{ratzeta}. \qed
 
 \subsection{Entropy}\label{sec:entropy}
 We aim to establish Theorem \ref{main} in this section. From previous sections, the main part is to construct a subhyperbolic rational map in $K(z)$ whose entropy is the logarithm of a  given weak Perron number $\lambda$. As aforementioned in \S\ref{sec:intro}, we first construct a matrix with leading eigenvalue $\lambda$ satisfying the properties that reflect the non-archimedean property, see \S\ref{sec:admissible}. The construction relies on some preliminaries in the graph theory introduced in \S\ref{sec:graph}. Moreover, the resulting matrix induces a piecewise linear map  and records the transition of the pieces in the domain, see  \S\ref{sec:linear}. Then, in \S\ref{sec:gluing}, we glue the pieces  to obtain a hyperbolic rational map in $K(z)$ and study its dynamics, which results in a proof of Theorem \ref{main}.
  In this section, we take the field $K$ as in Theorem \ref{main} and let $\mathfrak{u}\in K$ be a uniformizer of $K$. To simplify the notation, when $K$ is clear from the context, we write $|\cdot|$ instead of $|\cdot|_K$. Moreover, a closed disk $\overline{D}$ in $K$ is \emph{rational} if $\mathrm{diam}(\overline{D})$ is contained in the value group $|K\setminus\{0\}|$. 

\subsubsection{Augmented graphs}\label{sec:graph}
Let $G=(V, E)$ be a  finite directed graph, where $V:=V(G)$ is the set of vertices of $G$ and $E:=E(G)$  is the set of directed edges of $G$. Recall from \S\ref{sec:AM}  the adjacency matrix $A_G$ of $G$. 
 We say that $G$ is \emph{simple} if there is at most one directed edge with initial point $v$ and end point $v'$ for any two given vertices $v,v'\in V$; equivalently, $A_G$ is a (0,1)-matrix. One way to obtain a simple graph from a given finite directed graph $G_0$ is to construct augmented graphs by adding extra vertices:
 
 \begin{definition}\label{def:augmented}
 Given an integer $n\ge 1$, an \emph{$n$-augmented graph} $G$ of $G_0$ is a finite directed graph obtained by replacing each  $e\in E(G_0)$ with a directed path of length $n$ via adding $n-1$ extra vertices $v^e_1,\dots, v^e_{n-1}$ in between with $v^e_i\not=v^{e'}_j$ for any directed edges $e\not=e'$ of $G_0$ and any $1\le i\le j\le n-1$.
 \end{definition} 
 If $G$ is an $n$-augmented graph of $G_0$, then $\#E(G)=n\cdot\#E(G_0)$ and $V(G_0)$ is naturally identified to a subset of $V(G)$. Moreover, for each vertex $v^e_i, 1\le i\le n$, of $G$, there exist a unique $e_1\in E(G)$ having initial point  $v^e_i$ and a unique $e_2\in E(G)$ having end point $v^e_i$, where $v^e_i$ is as in Definition  \ref{def:augmented}.   
 
 The following result states the leading eigenvalue of the adjacency matrix of an augmented graph. Recall  from \S\ref{sec:AM} the Artin-Mazur zeta function $\zeta(G,t)=\det(I-tA_G)^{-1}$.
 
 \begin{lemma}\label{lem:simple}
 Let $G_0$ be a finite directed graph with $\lambda_0\ge 1$ as the leading eigenvalue of $A_{G_0}$, and given $n\ge 1$,  let $G$ be an $n$-augmented graph of $G_0$. Then  the  leading eigenvalue $\lambda\ge 1$ of $A_G$ satisfies that  $\lambda^n=\lambda_0$.
  \end{lemma}

\begin{proof} 
Observe that for any $m\ge 1$, the number $N(G_0,m)$ of loops of length $m$ in $G_0$ equals the number of loops of length $nm$ in $G$. Thus 
 \begin{align*}
\det(I-tA_G)^{-1}=\zeta(G,t)&=\exp\left(\sum_{m=1}^\infty N(G_0,m)\frac{t^{mn}}{nm}\right)=\exp\left(\frac{1}{n}\sum_{m=1}^\infty N(G_0,m)\frac{(t^n)^m}{m}\right)\\
 &=\exp\left(\frac{1}{n}\log\left(\det(I-t^nA_{G_0})^{-1}\right)\right)=\det(I-t^nA_{G_0})^{-1/n}. 
 \end{align*}
It follows that the leading eigenvalue $\lambda\ge 1$ of $A_G$ is an $n$-th root of $\lambda_0$. 
\end{proof}

\subsubsection{Admissible matrices}\label{sec:admissible}
Consider an $n\times n$ matrix  $A=(a_{ij})$ and for each $1\le i\le n$, set 
$$\mathcal{I}_A(i):=\{1\le j\le n, a_{ij}\not=0\}.$$
\begin{definition}\label{def:admissible}
    An $n\times n$ irreducible matrix $A$ is  \emph{admissible} if it satisfies the following properties:
    \begin{enumerate}
    \item (non-zero property) there are no rows or columns that are identically zero;
    \item (constant property) non-zero entries in each row are constant; 
    \item (containing property) if $\mathcal{I}_A(i_1)\cap\mathcal{I}_A(i_2)\not=\emptyset$ for some $1\le i_1<i_2\le n$, then either $\mathcal{I}_A(i_1)\subseteq\mathcal{I}_A(i_2)$ or $\mathcal{I}_A(i_2)\subseteq\mathcal{I}_A(i_1)$; and 
    \item[(4)] (Markov property) for each $1\le i\le n$, there exists an integer $m_i\ge 1$ such that the $i$-th row of $A^{m_i}$ contains at least two  non-zero entries. 
    \end{enumerate}
\end{definition}

Every weak Perron number can be archived as the leading eigenvalue  of an admissible $(0,1)$-matrix: 
\begin{proposition}\label{prop:perron-matrix}
Let $\lambda>1$ be a weak Perron number. Then there exists an admissible $(0,1)$-matrix with leading eigenvalue $\lambda$.
 \end{proposition}

 \begin{proof}
 Let $n_0\ge 1$ be an integer such that $\lambda^{n_0}$ is a Perron number. It follows from Lind's theorem \cite[Theorem 1]{Lind84} that there exists a non-negative integer matrix $B$ such that 
 \begin{itemize}
 \item the leading eigenvalue of $B$ is $\lambda^{n_0}$, and
 \item all entries of $B^{j_0}$ are positive integers for some $j_0\ge 1$. 
 \end{itemize}
Observe that  $B^{j_0}$ has leading eigenvalue $\lambda^{n_0j_0}$.

  Let $G_0$ be the adjacency graph of $B^{j_0}$.  
  Consider an $(n_0j_0)$-augmented graph $G$ of $G_0$ and the adjacency matrix $A_G$. By Lemma \ref{lem:simple}, we conclude that $A_G$ has leading eigenvalue $\lambda$. Observing that $A_G$ is a $(0,1)$-matrix, to complete the proof, we show that $A_G$ is admissible. 
  
Noting that $G_0$ is strongly connected, we conclude that any vertex of $G$ is an initial point and an end point of (possibly distinct) edges of $G$. Thus $A_G$ satisfies the non-zero property. Since the non-zero entries of $A_G$ are $1$, the constant property of $A_G$ holds. Let us check the containing property for $A_G$. For any vertex $v\in V(G)$, let $e_1,e_2\in E(G)$ such that the end points of $e_1$ and of $e_2$ are $v$. Denote by $v_1$ the initial point of $e_1$ and by $v_2$ the initial point of $e_2$.   If $v_1\not=v_2$, then $v\in V(G_0)\subseteq V(G)$. It follows that $e_1$ (resp. $e_2$) is the only edge with initial point $v_1$ (resp. $v_2$). Thus  $A_G$ satisfies the  containing property.  Finally, since  $A_{G_0}=B^{j_0}$ satisfies the Markov property, we conclude that $A_G$ also satisfies the Markov property. 
 \end{proof}
 
 \subsubsection{Piecewise linear maps}\label{sec:linear}
 Given an admissible $n\times n$ (0,1)-matrix $A=(a_{ij})$, we now construct a piecewise linear map whose adjacency matrix is $A$ in the sense of Definition \ref{def:compatible} below.
 
  Observe that $\mathcal{I}_A(i)$ induces equivalence relations on $L_A:=\{1,2,\dots,n\}$ as follows: 
  \begin{itemize} 
  \item Two elements in $L_A$ are \emph{equivalent of level $1$} if they are contained in $\mathcal{I}_A(i)$ for some $i\in L_A$. Denote by $[j]_1$ the equivalence class of level $1$ containing $j\in L_A$. Moreover, we call $i\in L_A$ an \emph{indicator} of $[j]_1$ if $j\in\mathcal{I}_A(i)$ and $\#\mathcal{I}_A(i)=\max\{\#\mathcal{I}_A(i'): j\in \mathcal{I}_A(i')\}$. 
\item  Inductively, two elements $j_1,j_2\in L_A$ are \emph{equivalent of level $\ell\ge 2$} if $[j_1]_{\ell-1}=[j_2]_{\ell-1}$ and  $j_1,j_2\in\mathcal{I}_A(i)$ for some $i\in L_A$ not an indicator of any $[j_1]_s=[j_2]_s$ for $1\le s\le \ell-1$. In this case, we say that $[j_1]_{\ell}$ is a \emph{successor} of $[j_1]_{\ell-1}$. 
Moreover, we call $i\in L_A$ an \emph{indicator} of $[j]_\ell$ if $j\in\mathcal{I}_A(i)$ and $\#\mathcal{I}_A(i)=\max\{\#\mathcal{I}_A(i'): j\in \mathcal{I}_A(i')\ \text{and}\ i'\ \text{is not an indicator of}\ [j]_s\ \text{for any}\ 1\le s\le \ell-1\}$.  
\end{itemize}

Denote by $\mathcal{E}_A$ the set of all equivalence classes induced by the above equivalent relations. If $[j]_{\ell}\in\mathcal{E}_A$ but $[j]_{\ell+1}$ does not exist, we call $j$ a \emph{terminal} of level $\ell$. Denote by $\mathcal{T}_A([j]_{\ell})$ the set of terminals $j'$ of level $\ell$ such that $[j']_{\ell}=[j]_{\ell}$. 
For $[j]_\ell\in\mathcal{E}_A$, denote by $\mathcal{S}_A([j]_\ell)$ the set of successors of $[j]_\ell$. A class in $\mathcal{E}_A$ is \emph{minimal} if it has no successors. Observe that  if $[j]_\ell\in\mathcal{E}_A$ is minimal, then $j$ is a terminal of level $\ell$. Moreover, each $j\in L_A$ is a terminal of some level. 

\begin{example}
Consider the admissible (0,1)-matrix $A=\begin{bmatrix}
   1&1&1\\
    1&1&1\\
    1&1&0
    \end{bmatrix}.$
    Then $[1]_1=[2]_1=[3]_1$ has indicators $1$ and $2$, and $[1]_2=[2]_2$ is minimal and has indicator $3$. So $\mathcal{E}_A=\{[1]_1,[1]_2\}$ and $\mathcal{S}_A([1]_1)=\{[1]_2\}$. Moreover, $3\in\mathcal{T}_A([1]_1)$ is a terminal of level $1$, and both $1$ and $2$ are terminals in $\mathcal{T}_A([1]_2)$ of level $2$. 
\end{example}

Although the indicators for a given class in $\mathcal{E}_A$ may not be unique, they induce a well-defined map 
$$\kappa_A: L_A\to\mathcal{E}_A,$$
sending $i\in L_A$ to the class in $\mathcal{E}_A$ that has an indicator $i$. We now use $\kappa_A$ to construct a natural piecewise linear map in the following sense: 
\begin{definition}\label{def:compatible}
 Given an $n\times n$ (0,1)-matrix $A=(a_{ij})$, a map $f:\cup_{i=1}^n\overline{D}_i\to K$ defined on $n$ distinct rational closed disks $\overline{D}_i, 1\le i\le n$ in $K$ is \emph{linearly compatible} to $A$ if the following hold:
 \begin{enumerate}
 \item $f|_{\overline{D}_i}$ is linear for each $1\le i\le n$, and 
  \item up to reindex $1\le i\le n$, $\overline{D}_j\subseteq f(\overline{D}_i)$ if and only if $a_{ij}=1$ for $1\le i,j\le n$. 
  \end{enumerate}
 \end{definition}
 
 \begin{lemma}\label{lem:matrix-map}
 Let $A$ be an admissible $n\times n$ (0,1)-matrix. Then there exists a map $f:\cup_{i=1}^n\overline{D}_i\to K$ linearly compatible to $A$. Moreover, we can choose $f$ so that 
 the following hold:
 \begin{enumerate}
 \item $\overline{D}_i\subsetneq\overline{D}(0,1)$ and $\overline{D}_i\cap \overline{D}(0,|\mathfrak{u}|)=\emptyset$ for each $1\le i\le n$.
 \item $\cup_{j\in\mathcal{I}_A(i)}\overline{D}_j\subsetneq f(\overline{D}_i)$ for each $1\le i\le n$. 
 \end{enumerate}
 \end{lemma}
 \begin{proof}
 Let $N\ge 1$ be the number of classes of level $1$ in $\mathcal{E}:=\mathcal{E}_A$. We consider $N$ pairwise disjoint rational closed disks $U_i, 1\le i\le N$ in $K$, 
 each corresponding to an equivalence class in $\mathcal{E}$ of level $1$. 
 Since the residue field of $K$ has at least two elements, we can choose $U_i\subsetneq\overline{D}(0,1)$ so that $U_i\cap  \overline{D}(0,|\mathfrak{u}|)=\emptyset$ for each $1\le i\le N$.  
 

 Inductively, in the rational closed disk $U$ corresponding to an element $[j]_\ell\in\mathcal{E}$, 
 we consider $\#\mathcal{S}_A([j]_\ell)+\#\mathcal{T}_A([j]_\ell)$ many pairwise disjoint rational closed subdisks  so that 
 \begin{itemize}
 \item $\#\mathcal{S}_A([j]_\ell)$ disks one-to-one correspond to the successors in $\mathcal{S}_A([j]_\ell)$; 
 \item the remaining $\#\mathcal{T}_A([j]_\ell)$  disks one-to-one correspond to the terminals in $\mathcal{T}_A([j]_\ell)$; and
 \item the union of disks corresponding to the terminals is properly contained in $U$. 
 \end{itemize}

 Thus, we totally obtain $\#\mathcal{E}+n$ many pairwise disjoint rational closed disks with $\#\mathcal{E}$ disks one-to-one corresponding to the classes in $\mathcal{E}$ and $n$ disks one-to-one corresponding to the terminals. 
 Denote by $\mathcal{U}$ the set of $\#\mathcal{E}$ disks for  $\mathcal{E}$, and let $\mathcal{A}$ be the set formed by the $n$ disks for the terminals.

 Moreover, we get natural bijections 
 $$\iota:\mathcal{E}\to\mathcal{U},$$
 sending each class in $\mathcal{E}$ to its corresponding rational closed disk, and 
 $$\mathcal{L}: L_A\to\mathcal{A},$$
 sending $j\in\mathcal{L}$ to the rational closed disk for $j$ as a terminal.

Now recall the map $\kappa:=\kappa_A: L_A\to\mathcal{E}$. For each $i\in L_A$, considering the rational closed disks $\mathcal{L}(i)\in\mathcal{A}$ and $\kappa(i)\in\mathcal{E}$, we let 
$$f_i: \mathcal{L}(i)\to\iota\circ\kappa(i)$$ 
be any linear and bijective map. 
Write $\overline{D}_i:=\mathcal{L}(i)$ and consider 
$$f: \cup_{i=1}^n \overline{D}_i\to K$$
such that $f|_{\overline{D}_i}=f_i$ for each $1\le i\le n$. Then $f$ is a map linearly compatible to $A$ satisfying the desired properties. 
 \end{proof}

\subsubsection{Gluing}\label{sec:gluing}

Let $f:\cup_{i=1}^n\overline{D}_i\to K$ be a map in  Lemma \ref{lem:matrix-map} for  an admissible $n\times n$ (0,1)-matrix $A$. We now glue the pieces of $f$ to obtain hyperbolic rational map in $K(z)$ whose Julia set is contained in the domain of $f$. The gluing process employs the following reformulation of \cite[Theorem 4.2]{Nopal23}. Denote by $\mathbb{C}_K$ be the completion of an algebraical closure of $K$. 

\begin{lemma}\label{thm:surgery}
Let $B_1,B_2,\dots,B_s$ be  pairwise disjoint rational closed disks in $\mathbb{C}_K$. For each $1\le i\le s$, pick $a_i\in B_i$ and set $\rho_i:=\min_{i\not=j}\{|a_i-a_j|\}$. Choose $b_i\in\mathbb{C}_K$ such that $\mathrm{diam}(B_i)<|b_i|<\rho_i$. Consider rational maps $g_1,g_2,\dots,g_s$ in $\mathbb{C}_K(z)$ such that 
\begin{enumerate}
\item[$\mathrm{(i)}$] $g_i(B_i)$ is a disk in $\mathbb{C}_K$ for each  $1\le i\le s$; and 
\item[$(\mathrm{ii})$] $g_i(\cup_{i=1}^sB_i)$ is uniformly bounded for all $1\le i\le s$.
\end{enumerate}
Then for any $\epsilon>0$, there exists $M_0>1$ such that for all $M>M_0$, the map 
$$F(z)=\sum_{i=1}^sg_i(z)h_i(z),$$
where 
$$h_i(z)=\frac{1}{1-\left(\frac{z-a_i}{b_i}\right)^M}$$
satisfies the following: for each $1\le i\le s$, 
\begin{enumerate}
\item $F(B_i)=g_i(B_i)$ and in fact $|F(z)-g_i(z)|<\epsilon$ for any $z\in B_i$; and
\item $\deg_{B_i}F=\deg_{B_i}g_i$
\end{enumerate}
\end{lemma}

Set $\overline{D}_0=\overline{D}(0,|\mathfrak{u}|)$ and define 
 $$f_0(z):\overline{D}_0\to\overline{D}_0$$
 to be the identity map. For $1\le i\le n$, write 
 $$f_i:=f|_{\overline{D}_i}.$$ 
Note that the maps $f_i, 0\le i\le n,$ are in fact well-defined on $\mathbb{C}_K$. For each $0\le i\le n$, pick $a_i\in \overline{D}_i$ and choose $x_i\in K$ such that $|x_i|=\mathrm{diam}(\overline{D}_i)$. Set 
$$b_i:=\frac{x_i}{\sqrt{\mathfrak{u}}}.$$
Then 
each $b_i$ is contained in the quadratic extension $K(\sqrt{\mathfrak{u}})$ of $K$. Let $M\ge 1$ be a sufficiently large even number to be determined later. 
 For each $0\le i\le n$, we consider 
$$h_i(z):=\frac{1}{1-\left(\frac{z-a_i}{b_i}\right)^M}.$$
and define 
$$F(z):=\sum_{i=0}^{n}f_i(z)h_i(z).$$
Since $M$ is even, we conclude that $h_i\in K(z)$ and hence $F\in K(z)$. 

\begin{proposition}\label{prop:F}
For a sufficiently large even integer $M\ge 1$, the map $F$ is a hyperbolic map such that the following hold:
\begin{enumerate}
    \item $\mathcal{J}(F,K)\not=\emptyset$ is compact.
    \item $\mathcal{J}(F,K)\subseteq\cup_{i=1}^n\overline{D}_i$ and $\mathcal{J}(F,K)\cap \overline{D}_i\not=\emptyset$ for each $1\le i\le n$.
    \item The adjacency matrix of $F$ on $\overline{D}_i,1\le i\le n,$ is $A$
\end{enumerate}
\end{proposition}

The proof of Proposition \ref{prop:F} is provided after we state the following two lemmas on the Fatou set and Julia set of $F$. 
\begin{lemma}\label{lem:M}
There exists a sufficiently large even integer $M\ge 1$ such that 
 $$F(\mathbb{P}^1(K)\setminus(\cup_{i=1}^n \overline{D}_i))\subseteq\overline{D}_0.$$
\end{lemma}
\begin{proof}
Note that $F(\infty)=0\in\overline{D}_0$. Now let us pick $z_0\in K\setminus(\cup_{i=1}^n \overline{D}_i)$ and estimate $|f_i(z_0)h_i(z_0)|$ for each $1\le i\le n$. Since $f_i$ is linear, we can write $f_i(z)=\alpha_iz+\beta_i$ for any $z\in\mathbb{C}_K$ for some $\alpha_i\in K\setminus\{0\}$ and $\beta_i\in K$. It follows that 
$$f_i(z_0)h_i(z_0)=\frac{\alpha_iz_0+\beta_i}{1-\left(\frac{z_0-a_i}{b_i}\right)^M}.$$
From the choice of $b_i$, we have that 
$$\left|\frac{z_0-a_i}{b_i}\right|\ge\frac{1}{|\sqrt{\mathfrak{u}}|}>1.$$
Thus 
$$|f_i(z_0)h_i(z_0)|=|\alpha_iz_0+\beta_i|\left|\frac{b_i}{z_0-a_i}\right|^M.$$

If $|\alpha_iz_0|\le|\beta_i|$, then  
$$|f_i(z_0)h_i(z_0)|\le|\beta_i||\sqrt{\mathfrak{u}}|^M.$$
Thus there exists $M_i^{(0)}\ge 1$ such that $|f_i(z_0)h_i(z_0)|\le |\mathfrak{u}|$ for any $M\ge M_i^{(0)}$.

We now consider the case where $|\alpha_iz_0|> |\beta_i|$. Then 
$$|f_i(z_0)h_i(z_0)|\le|\alpha_iz_0|\left|\frac{b_i}{z_0-a_i}\right|^M=\left|\frac{\alpha_iz_0}{z_0-a_i}\right| |b_i| \left|\frac{b_i}{z_0-a_i}\right|^{M-1}\le \left|\frac{\alpha_iz_0}{z_0-a_i}\right| |b_i||\sqrt{\mathfrak{u}}|^{M-1}.$$
If $|z_0|>|a_i|$, then $|z_0-a_i|\le |z_0|$ and hence 
$$|f_i(z_0)h_i(z_0)|\le|\alpha_ib_i||\sqrt{\mathfrak{u}}|^{M-1}.$$
If $|z_0|<|a_i|$, then $|z_0-a_i|=|a_i|$ and hence $|(\alpha_iz_0)/(z_0-a_i)|<|\alpha_i|$; so we also obtain that 
$$|f_i(z_0)h_i(z_0)|\le|\alpha_ib_i||\sqrt{\mathfrak{u}}|^{M-1}.$$
If $|z_0|=|a_i|$, since $z_0\not\in\overline{D}_i$, we have that $|z_0-a_i|\ge\mathrm{diam}(\overline{D}_i)$; so we obtain that 
$$|f_i(z_0)h_i(z_0)|\le\frac{\left|\alpha_ia_ib_i\right|}{\mathrm{diam}(\overline{D}_i)}|\sqrt{\mathfrak{u}}|^{M-1}.$$
Thus in this case, there exist $M_i^{(1)}\ge 1$ such that $|f_i(z_0)h_i(z_0)|\le |\mathfrak{u}|$ for any $M\ge M_i^{(1)}$.

Set $M_i=\max\{ M_i^{(0)},  M_i^{(1)}\}$. Then $|f_i(z_0)h_i(z_0)|\le |\mathfrak{u}|$ for any $M\ge M_i$. Now we take 
$$M\ge\max\{M_i:1\le i\le n\}$$
to be an even integer and conclude that, for any $z_0\in\mathbb{P}^1(K)\setminus(\cup_{i=1}^n \overline{D}_i)$, 
$$|F(z_0)|\le\max\{|f_i(z_0)h_i(z_0)|:1\le i\le n\}\le|\mathfrak{u}|.$$
Thus the conclusion holds. 
\end{proof}

We now pick $M$ as in Lemma \ref{lem:M}. Since  $F(\overline{D}_0)=f_0(\overline{D}_0)=\overline{D}_0$, we have that $\overline{D}_0$  is contained in the Fatou set of $F$. Then $\mathbb{P}^1(K)\setminus(\cup_{i=1}^n \overline{D}_i)$ is also contained in the Fatou set of $F$ by Lemma \ref{lem:M}. Hence 
$$\mathcal{J}(F,K)\subseteq\cup_{i=1}^n \overline{D}_i.$$

\begin{lemma}\label{lem:F-J}
The Julia set $\mathcal{J}(F,K)$ intersects $\overline{D}_i$ for each $1\le i\le n$.
\end{lemma}
\begin{proof}
Pick any $1\le i_0\le n$. By Lemma \ref{lem:matrix-map} (2), we have that $F(\overline{D}_{i_0})\setminus(\cup_{i=1}^n \overline{D}_i)\not=\emptyset$.
Then by Lemma \ref{lem:M}, we conclude that 
$$F^{\circ 2}(\overline{D}_{i_0})\cap\overline{D}_0\not=\emptyset.$$
Again by Lemma \ref{lem:matrix-map} (2), we have that
$$F^{\circ 2}(\overline{D}_{i_0})\cap(\cup_{i=1}^n \overline{D}_i)\not=\emptyset.$$
Thus there exist $x_{i_0},y_{i_0}\in \overline{D}_{i_0}$ such that $|F^{\circ 2}(x_{i_0})- F^{\circ 2}(y_{i_0})|\ge 1$. 

Now given $1\le j\le n$, for any $m\ge 1$, consider a closed disk $\overline{D}_{j_m}\subseteq F^{\circ m}(\overline{D}_j)$ with $1\le j_m\le n$. Let $\overline{D}_{j,m}\subseteq\overline{D}_j$ be a closed disk where $F^{\circ m}$ is bijective to $\overline{D}_{j_m}$. The existence of such a $\overline{D}_{j,m}$ follows from the fact that $F$ is linear in each $\overline{D}_i$ for $1\le i\le n$. Then passing to a subsequence in $m$, we obtain that as $m\to\infty$,
$$\mathrm{diam}(\overline{D}_{j,m})\to 0.$$
Consider $F^{\circ m}$-preimages $x_{j,m},y_{j,m}$ in $\overline{D}_{j,m}$ for the points $x_{j_m}, y_{j_m}$ in $\overline{D}_{j_m}$, respectively. We conclude that, as $m\to\infty$, 
$$|x_{j,m}-y_{j,m}|\to 0,$$
but 
$$|F^{\circ {(m+2)}}(x_{j,m})-F^{\circ {(m+2)}}(y_{j,m})|=|F^{\circ 2}(x_{j_m})-F^{\circ 2}(y_{j_m})|\ge 1.$$
Thus $\{F^{\circ n}\}_{n\ge 1}$ is not equicontinuous in $\overline{D}_j$, equivalently, $\mathcal{J}(F,K)\cap\overline{D}_j\not=\emptyset$.  
 Hence the conclusion holds.
\end{proof}

\begin{proof}[Proof of Proposition \ref{prop:F}]
By Lemmas \ref{lem:matrix-map} and \ref{lem:F-J}, it suffices to check the hyperbolicity of  $F$. By Lemma \ref{thm:surgery} (2) and the linearity of $f_i, 1\le i\le n$, we conclude that in each $\overline{D}_i$ the map $F$ has degree $1$ and hence has no critical points. So $F$ has no critical points in $\mathcal{J}(F,K)$. Thus, $F$ is hyperbolic. 
\end{proof}

\begin{proof}[Proof of Theorem \ref{main}]
Let $\lambda\ge 1$ be a weak Perron number. If $\lambda=1$, we consider a rational map $\phi\in K(z)$ of degree at least $2$ having a good reduction. Then $J(\phi,K)=\emptyset$ and hence $\phi$ is a hyperbolic map with $\mathfrak{h}(\phi,K)=0=\log 1$. If $\lambda>1$, by Proposition \ref{prop:perron-matrix}, we consider an admissible $(0,1)$-matrix $A$ with leading eigenvalue $\lambda$. Then the map $\phi:=F$ in Proposition \ref{prop:F} is a hyperbolic rational map in $K(z)$. Moreover,  by Proposition \ref{prop:finite}, $(\mathcal{J}(\phi,K),\phi)$ is topologically conjugate to $(\Sigma_A,\sigma)$, which implies that $\mathfrak{h}(\phi,K)=\log\lambda$. 

Conversely, let $\phi\in K(z)$ be a subhyperbolic rational map in $K(z)$ with compact $\mathcal{J}(\phi,K)$ such that $\mathfrak{h}(\phi,K)\not=0$. Then there exist a matrix $B$ with non-negative entries and finitely many numbers $\ell_1,\dots,\ell_{m_0}$ such that 
$$\zeta_\phi(K,t)=\frac{\prod_{m=1}^{m_0}(1-t^{\ell_m})}{\det(I-tB)},$$
see \S\ref{sec:pr}. Noting that  the numerator and denominator of $\zeta_\phi(K,t)$ have no common zeros differing from roots of unity, we conclude that the reciprocal of the convergence radius of $\zeta_\phi(K,t)$ is a weak Perron number $\lambda_B>1$; in fact, $\lambda_B$ is the  leading eigenvalue of $B$. Hence $\mathfrak{h}(\phi,K)=\log\lambda_B$. 
\end{proof}

\section{Examples}\label{sec:example}
In this section, we compute concrete examples for the Artin-Mazur zeta function and the entropy. 

\subsection{Hyperbolic rational maps with prescribed entropy}
We first apply our construction in Subsection~\ref{sec:entropy} to give two  hyperbolic rational maps with prescribed entropies. 

\begin{example}[A hyperbolic rational map $\phi\in\mathbb{Q}_2(z)$ with $\exp(\mathfrak{h}(\phi,\mathbb{Q}_2)=(1+\sqrt{5})/2$]\label{ex:mapA}
Let $A=\begin{bmatrix}
   1&1\\
    1&0
    \end{bmatrix}.$
    Then $A$ is an admissible (0,1)-matrix with leading eigenvalue $(1+\sqrt 5)/2$. Let us consider the disks $\overline{D}(1,|8|)$ and $\overline{D}(3,|4|)$ in $\mathbb{Q}_2$, and define the maps 
    \begin{itemize}
    \item $f_1:\overline{D}(1,|8|)\to \overline{D}(1,|2|)$, sending $z\in \overline{D}(1,|8|)$ to $(z+3)/4\in\overline{D}(1,|2|)$, and 
    \item  $f_2:\overline{D}(3,|4|)\to \overline{D}(1,|4|)$, sending $z\in \overline{D}(3,|4|)$ to $z-2\in\overline{D}(1,|4|)$. 
    \end{itemize}
    Moreover, let  $f_0:\overline{D}(0,|2|)\to \overline{D}(0,|2|)$ be the identity map. See Figure \ref{f:10}. 
     \begin{figure}[h]
     \includegraphics[width=0.5\textwidth]{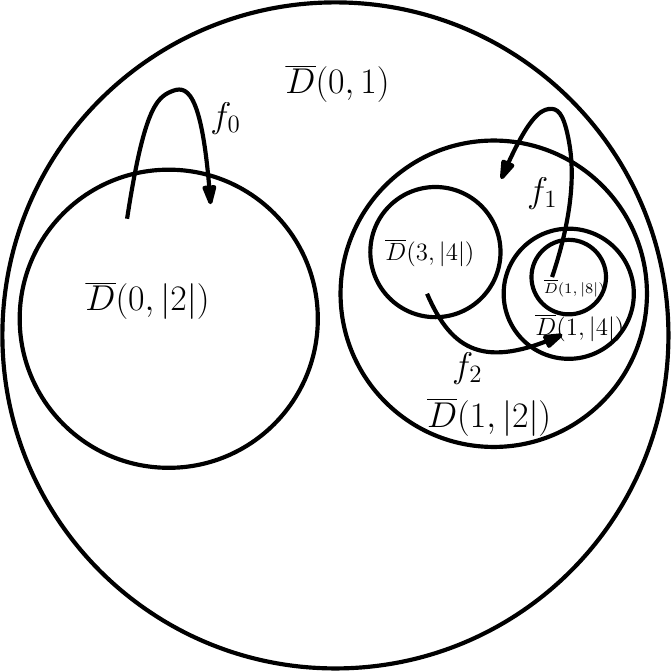}
	\caption{The linear maps $f_1,f_2$ and $f_0$ in Example \ref{ex:mapA}.}
    \label{f:10}
\end{figure}
Computing a large even integer $M$ as in the proof of Lemma \ref{lem:M}, we can take $M=14$.  Then we obtain a desired hyperbolic rational map
$\phi\in \mathbb{Q}_2(z)$ defined by 
\begin{align*}
\phi(z)&=\frac{f_1(z)}{1-\left(\frac{z-1}{4\sqrt{2}}\right)^{14}}+\frac{f_2(z)}{1-\left(\frac{z-3}{2\sqrt{2}}\right)^{14}}+\frac{f_0(z)}{1-\left(\frac{z}{\sqrt{2}}\right)^{14}}\\
&=\frac{2^{33}(z+3)}{2^{35}-(z-1)^{14}}+\frac{2^{21}(z-2)}{2^{21}-(z-3)^{14}}+\frac{2^7z}{2^7-z^{14}}.
\end{align*}
\end{example}

 \begin{example}[A hyperbolic rational map $\phi\in\mathbb{Q}_2(z)$ with $\mathfrak{h}(\phi,\mathbb{Q}_2)=0$ and $\mathcal{J}(\phi,\mathbb{Q}_2)\not=\emptyset$]\label{ex:mapB}
 Let $A=\begin{bmatrix}
   0&1\\
    1&0
    \end{bmatrix}.$
    Then $A$ is an admissible (0,1)-matrix with leading eigenvalue $1$. Let us consider the disks $\overline{D}(1,|8|)$ and $\overline{D}(3,|8|)$ in $\mathbb{Q}_2$, and define the maps 
    \begin{itemize}
    \item $f_1:\overline{D}(1,|8|)\to \overline{D}(3,|4|)$, sending $z\in \overline{D}(1,|8|)$ to $(z+5)/2\in\overline{D}(3,|4|)$, and 
    \item  $f_2:\overline{D}(3,|8|)\to \overline{D}(1,|4|)$, sending $z\in \overline{D}(3,|8|)$ to $(z-1)/2\in\overline{D}(1,|4|)$. 
    \end{itemize}
    Moreover, let  $f_0:\overline{D}(0,|2|)\to \overline{D}(0,|2|)$ be the identity map. See Figure \ref{f:20}. 
      \begin{figure}[h]
     \includegraphics[width=0.5\textwidth]{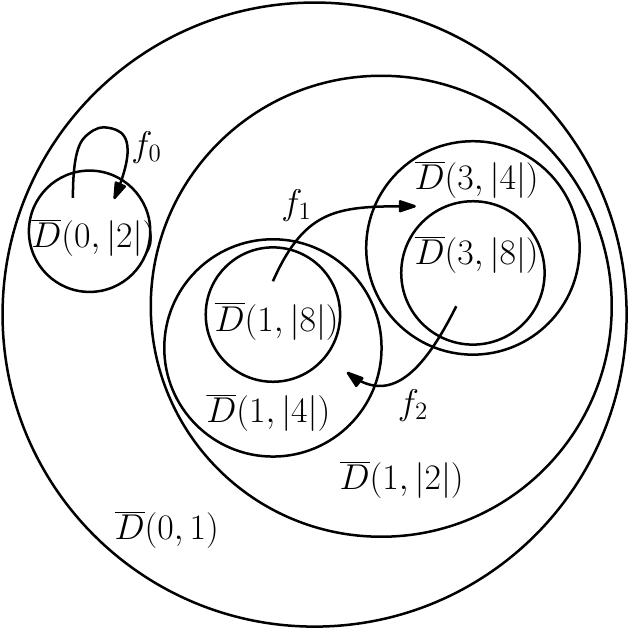}
	\caption{The linear maps $f_1,f_2$ and $f_0$ in Example \ref{ex:mapB}.}
    \label{f:20}
\end{figure}
    Computing a large even integer $M$ as in the proof of Lemma \ref{lem:M}, we can take $M=6$.  Then we obtain a  hyperbolic rational map
$\phi\in \mathbb{Q}_2(z)$ defined by 
\begin{align*}
\phi(z)&=\frac{f_1(z)}{1-\left(\frac{z-1}{4\sqrt{2}}\right)^{6}}+\frac{f_2(z)}{1-\left(\frac{z-3}{4\sqrt{2}}\right)^{6}}+\frac{f_0(z)}{1-\left(\frac{z}{\sqrt{2}}\right)^{6}}\\
&=\frac{2^{14}(z+5)}{2^{15}-(z-1)^{6}}+\frac{2^{14}(z-1)}{2^{15}-(z-3)^{6}}+\frac{2^3z}{2^3-z^{6}}.
\end{align*}
In $\overline{D}(1,|8|)$, since $\phi(\overline{D}(3,|4|)\setminus\overline{D}(3,|8|))\subseteq \overline{D}(0,|2|)$, we have that 
$$\overline{D}(1,|8|)\subsetneq\phi^{\circ 2}(\overline{D}(1,|8|))\subseteq \overline{D}(1,|4|)\cup\overline{D}(0,|2|).$$
Note that $\phi^{\circ 2}$ is a scaling map in $\overline{D}:=\phi^{-2}(\overline{D}(1,|4|))\cap \overline{D}(1,|8|)$. Thus $\overline{D}$ contains a fixed point, say $x_0$, of $\phi^{\circ 2}$. 
Since $$\phi^{\circ 2}(\overline{D}(1,|4|)\setminus\overline{D}(1,|8|))\subseteq \overline{D}(0,|2|),$$ we conclude that any point in $\overline{D}(1,|8|)\setminus\{x_0\}$ is eventually mapped into $\overline{D}(0,|2|)$ under $\phi$ and hence is contained in $\mathcal{F}(\phi,\mathbb{Q}_2)$. Thus $\mathcal{J}(\phi,\mathbb{Q}_2)$ coincides with the two-cycle $\{x_0,\phi(x_0)\}$.

\end{example}

\subsection{Artin-Mazur zeta function for polynomials} 

Given any prime number $p\ge 2$, let $K=\widehat{\mathbb{Q}_p^{un}}$ be the completion of the maximal unramified extension of $\mathbb{Q}_p$. Consider the polynomial $\phi\in K[z]$, acting as a self-map on $\mathbb{P}^1(K)$, defined by 
$$\phi(z)=\frac{z(z-1)^2}{8p}.$$
We have the following observations:
\begin{itemize}
\item the critical points of $\phi$ in $K$ are $1$ and $1/3$;
\item the point $0$ is a repelling fixed point and hence $0\in\mathcal{J}(\phi, K)$; 
\item $\phi(1)=0$ and hence $1\in\mathcal{J}(\phi, K)$; 
\item $\phi^{\circ n}(1/3)\to\infty$, as $n\to\infty$, and hence $1/3\in\mathcal{F}(\phi,K)$; and 
\item $\phi^{\circ n}(z)\to\infty$, as $n\to\infty$, for any $z\not\in D(0,1)\cup D(1,1)$. 
\end{itemize}
In particular, $\phi$ is subhyperbolic. Moreover, $\mathcal{J}(\phi,K)$ is compact by \cite[Theorem 4.8]{Hsia96}. 

We now compute $\zeta_{\phi}(K,t)$ and $\mathfrak{h}(\phi,K)$ depending on whether $p$ is odd or even. In the former case, the critical point $1$ is tame, and in the latter case, the critical point $1$ is wild. 

\textbf{Case 1: $p\not=2$.}
For each $n\ge 1$, consider the sets $A_n:=\overline D(0,|p^{n}|)\setminus D(0,|p^{n}|)$ and $\tilde A_n:=\overline D(1,|p^{n}|)\setminus D(1,|p^{n}|)$ 

In $A_n$, since $\phi(\overline D(0,|p|))=\overline{D}(0,1)=\overline{D}(1,1)$ and $\phi$ has degree $1$ on $\overline D(0,|p|)$, we conclude that
\begin{itemize}
\item  in $A_1$, there is unique closed disk $\overline D_1$ such that $\phi(\overline D_1)=\overline D(1,|p|)$; 
\item for each $n\ge 2$, in $A_n$,  there is unique closed disk $\overline D_n$ such that $\phi(\overline D_n)=\overline D_{n-1}$;
\item for each $n\ge 1$, $\phi^{\circ (n+1)}(A_n\setminus\overline D_n)=K\setminus(D(0,1)\cup D(1,1))$ and hence $A_n\setminus D_n\subseteq\mathcal{F}(\phi,K)$. 
\end{itemize}
It follows that 
$$D(0,1)\cap\mathcal{J}(\phi,K)\subseteq\bigcup_{n\ge 1}\overline D_n\cup\{0\}.$$
In $\tilde A_n$, since $\phi(\overline{D}(1,|p^{n}|))=\overline{D}(1,|p^{2n-1}|)$ and $\phi$ has degree $2$ on $\overline D(1,|p^{2n}|)$, we conclude that 
\begin{itemize}
\item for each $n\ge 1$, in $\tilde A_n$,  there are exactly two distinct  closed disk $\overline B_n^{+}$ and $\overline B_n^{-}$ such that $\phi(\overline B_n^{\pm})=\overline D_{2n-1}$. 
\end{itemize}
It follows that 
$$D(1,1)\cap\mathcal{J}(\phi,K)\subseteq\bigcup_{n\ge 1}(\overline B_n^+\cup \overline B_n^-)\cup\{1\}.$$
Thus 
$$\mathcal{J}(\phi,K)\subseteq\bigcup_{n\ge 1}(\overline D_n\cup\overline B_n^+\cup \overline B_n^-)\cup\{1\}\cup\{0\}.$$
Moreover, we obtain countable symbols  $\overline{D}_n$, $\overline B_n^+$ and $\overline B_n^-$ (see Figure \ref{f:2}).
   \begin{figure}[h]
     \includegraphics[width=0.8\textwidth]{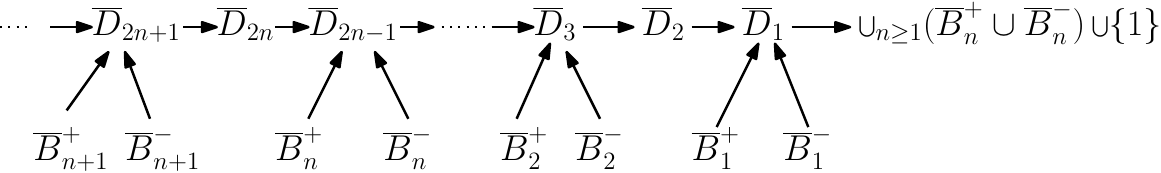}
	\caption{The portrait of $\overline{D}_n$, $\overline B_n^+$ and $\overline B_n^-$ under $\phi$.}
    \label{f:2}
\end{figure}

We now identify $\overline{D}_{2n-1},n\ge1$ as a symbol $\alpha^+$, $\overline{D}_{2n},n\ge1$ as a symbol $\alpha^-$, $\overline{B}_n^+,n\ge1$ as a symbol $\beta^+$, and $\overline{B}_n^-,n\ge1$ as a symbol $\beta^-$,  exchanging the disks $\overline{B}_n^-$ and $\overline{B}_n^+$ for some $n$ if necessary. Then adding the critical point $\{1\}$ and the critical value $\{0\}$, we obtain a finite graph (see Figure \ref{f:3}) with adjacency matrix $A$ given by 

\[
A:=\begin{blockarray}{&ccccccc}
  &\{1\} & \{0\} & \alpha^+&\alpha^-&\beta^+&\beta^- \\
\begin{block}{c(ccccccc)}
 \{1\}& 0& 1&0&0&0&0\\
 \{0\}&0&1&0&0&0&0\\
 \alpha^+&1&0&0&1&1&1\\
  \alpha^-&0&0&1&0&0&0\\
 \beta^+&0&0&1&0&0&0\\
 \beta^-&0&0&1&0&0&0\\
\end{block}
\end{blockarray}\ .
 \]
  \begin{figure}[h]
     \includegraphics[width=0.45\textwidth]{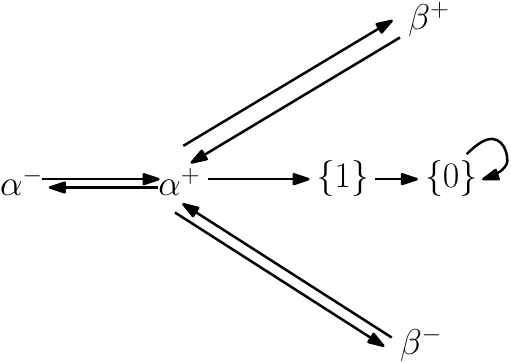}
	\caption{A finite graph for $\phi$ when $p\not=2$.}
    \label{f:3}
\end{figure}

Observe that the two-cycle of periodic points $(\alpha^+,\alpha^-,\alpha^+,\alpha^-,\dots)$ and  $(\alpha^-,\alpha^+,\alpha^-,\alpha^+,\dots)$ in the symbolic space $(\Sigma_A,\sigma)$ induced by the above symbols is the only periodic cycle not corresponding to a periodic cycle in $\mathcal{J}(\phi,K)$. Thus we conclude that 
$$\zeta_\phi(K,t):=\frac{1-t^2}{\det(I-tA)}=\frac{1-t^2}{(1-t)(1-3t^2)}=\frac{1+t}{1-3t^2}.$$
Hence
$\mathfrak{h}(\phi,K)=\log\sqrt3$. 

\textbf{Case 2: $p=2$.} For each $n\ge 1$, consider the sets $A_n:=D(0,|p^{4n-4}|)\setminus D(0,|p^{4n}|)$ and $\tilde A_n:=D(1,|p^{2n-2}|)\setminus D(0,|p^{2n}|)$.

In $A_n$, since $\phi(\overline D(0,|p^4|))=\overline{D}(0,1)=\overline{D}(1,1)$ and $\phi$ has degree $1$ on $\overline D(0,|p^4|)$, we conclude that
\begin{itemize}
\item  in $A_1$, there are unique closed disks $\overline D_1\subset\overline{U}_1$ such that $\phi(\overline D_1)=\overline D(1,|p^3|)$ and $\phi(\overline U_1)=\overline D(1,|p^2|)$; 
\item for each $n\ge 2$, in $A_n$,  there are unique closed disks $\overline D_n\subset\overline{U}_n$ such that $\phi(\overline D_n)=\overline D_{n-1}$ and $\phi(\overline U_n)=\overline U_{n-1}$;
\item for each $n\ge 1$, $\phi^{\circ (n+1)}(A_n\setminus\overline U_n)=K\setminus(D(0,1)\cup D(1,1))$ and hence $A_n\setminus U_n\subseteq\mathcal{F}(\phi,K)$. 
\end{itemize}
It follows that 
$$D(0,1)\cap\mathcal{J}(\phi,K)\subseteq\bigcup_{n\ge 1}\overline U_n\cup\{0\}.$$

In $\tilde A_n$, since $\phi(\overline{D}(1,|p^{2n}|))=\overline{D}(0,|p^{4n-4}|)$ and $\phi$ has degree $2$ on $\overline D(1,|p^{2n}|)$, we conclude that 
\begin{itemize}
\item for each $n\ge 2$, in $\tilde A_n$,  there are exactly two distinct closed disks $\overline B_n^{+}$ and $\overline B_n^{-}$ such that $\phi(\overline B_n^{\pm})=\overline D_{n-1}$;
\item  in $\tilde A_1$, there are exactly two distinct closed disks $\overline B_1^{+}$ and $\overline B_1^{-}$ such that $\phi(\overline B_1^{\pm})=\overline D(1,|p^3|)$;
\item  in $\tilde A_1$, there are exactly two distinct closed disks $\overline V^{+}$ and $\overline V^{-}$ such that $\overline B_1^{\pm}\subseteq\phi(\overline V_1^{\pm})$ and $\mathrm{diam}(\phi(\overline V^{\pm}))=p\cdot\mathrm{diam}(\overline B_1^{\pm})$;
\end{itemize}
It follows that 
$$\overline V^{\pm}\subsetneq\phi(\overline V^+)=\phi(\overline V^-).$$
Moreover, 
$$\tilde A_1\cap\mathcal{J}(\phi,K)\subseteq\overline B_1^+\cup \overline B_1^-\cup\overline V^+\cup \overline V^-.$$

In $\overline U_n$,  since $\phi(\overline U_1)=\overline D(1,|p^2|)$, we conclude that 
\begin{itemize}
\item  in $\overline U_1$, there is a unique closed disk $\overline W_1\not=\overline D_1$ such that $\phi(\overline W_1)=\phi(\overline V_1^{\pm})$; 
\item for each $n\ge 2$, in $\overline U_n$,  there is a unique closed disk $\overline W_n\not=\overline D_n$ such that $\phi(\overline W_n)=\overline W_{n-1}$;
\end{itemize}
Then in $\tilde A_n$, we conclude that 
\begin{itemize}
\item for each $n\ge 2$, in $\tilde A_n$,  there are exactly two distinct closed disks $\overline X_n^{+}$ and $\overline X_n^{-}$, distinct from $\overline{B}_n^{\pm}$ such that $\phi(\overline X_n^{\pm})=\overline W_{n-1}$. 
\end{itemize}
It is easy to check that 
$$\mathcal{J}(\phi,K)\subseteq\bigcup_{n\ge 1}(\overline D_n\cup\overline W_n\cup\overline B_n^+\cup \overline B_n^-\cup\overline X_n^+\cup \overline X_n^-)\cup\overline V^+\cup \overline V^-\cup\{1\}\cup\{0\}.$$
Moreover, we obtain countable symbols  $\overline{D}_n, \overline W_n$, $\overline B_n^\pm$,$\overline X_n^{\pm}$ and $\overline V^{\pm}$ (see Figure \ref{f:5}).
  \begin{figure}[h]
    \includegraphics[width=0.75\textwidth]{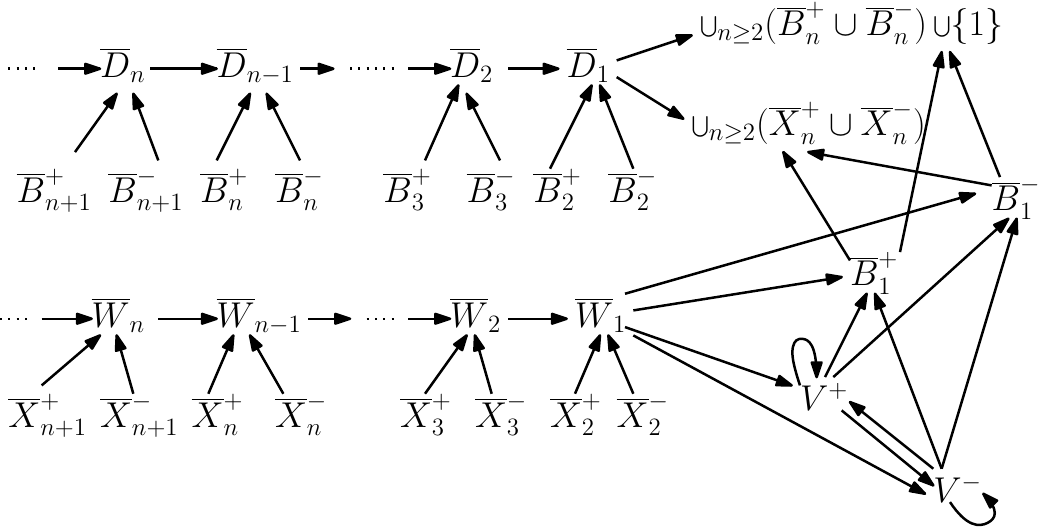}
	\caption{The portrait of $\overline{D}_n, \overline W_n$, $\overline B_n^\pm$,$\overline X_n^{\pm}$ and $\overline V^{\pm}$  under $\phi$.}
    \label{f:5}
\end{figure}

We now identify $\overline{D}_{n},n\ge1$ as a symbol $\alpha$, $\overline{W}_{n},n\ge1$ as a symbol $\omega$, $\overline{B}_n^+,n\ge2$ as a symbol $\beta^+$, $\overline{B}_n^-,n\ge2$ as a symbol $\beta^-$, $\overline{B}_1^+$ as a symbol $\beta_1^+$, $\overline{B}_1^-$ as a symbol $\beta_1^-$, $\overline{X}_n^+,n\ge2$ as a symbol $\theta^+$, $\overline{X}_n^-,n\ge2$ as a symbol $\theta^-$,  $\overline V^+$ as a symbol $v^+$, $\overline V^-$ as a symbol $v^-$, exchanging the disks $\overline{B}_n^-$ and $\overline{B}_n^+$ for some $n$ and exchanging the disks $\overline{X}_m^-$ and $\overline{X}_m^+$ for some $m$ if necessary. Then adding the critical point $\{1\}$ and the critical value $\{0\}$, we obtain a finite graph (see Figure \ref{f:6}) with adjacency matrix $A$ given by 

\[
A:=\begin{blockarray}{&cccccccccccc}
  &\{1\} & \{0\} & \alpha&\beta^+&\beta^-& \beta_1^+& \beta_1^-&\omega&\theta^+&\theta^-&v^+&v^- \\
\begin{block}{c(cccccccccccc)}
 \{1\}& 0& 1&0&0&0&0&0&0&0&0&0&0\\
 \{0\}&0&1&0&0&0&0&0&0&0&0&0&0\\
 \alpha&1&0&1&1&1&0&0&0&1&1&0&0\\
 \beta^+&0&0&1&0&0&0&0&0&0&0&0&0\\
 \beta^-&0&0&1&0&0&0&0&0&0&0&0&0\\
  \beta_1^+&1&0&0&1&1&0&0&0&1&1&0&0\\
 \beta_1^-&1&0&0&1&1&0&0&0&1&1&0&0\\
  \omega&0&0&0&0&0&1&1&1&0&0&1&1\\
   \theta^+&0&0&0&0&0&0&0&1&0&0&0&0\\
 \theta^-&0&0&0&0&0&0&0&1&0&0&0&0\\
  v^+&0&0&0&0&0&1&1&0&0&0&1&1\\
 v^-&0&0&0&0&0&1&1&0&0&0&1&1\\
\end{block}
\end{blockarray}\ .
 \]
   \begin{figure}[h]
     \includegraphics[width=0.7\textwidth]{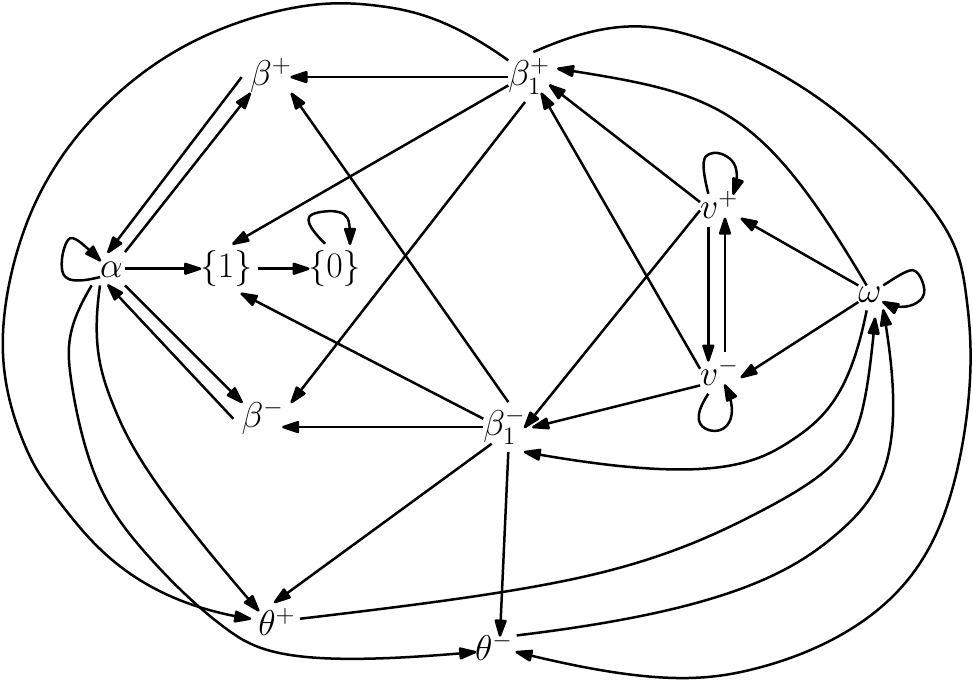}
	\caption{A finite graph for $\phi$ when $p=2$.}
    \label{f:6}
\end{figure}

 Observe that the two fixed  points $(\alpha,\alpha,\alpha,\dots)$ and  $(\omega,\omega,\omega,\dots)$ in the symbolic space $(\Sigma_A,\sigma)$ induced by the above symbols are the only periodic cycles not corresponding to periodic cycles in $\mathcal{J}(\phi,K)$. Thus we conclude that 
$$\zeta_\phi(K,t):=\frac{(1-t)^2}{\det(I-tA)}=\frac{(1-t)^2}{(1-t)^2(1-3t)}=\frac{1}{1-3t}.$$
Hence
$\mathfrak{h}(\phi,K)=\log 3$. 
\bibliographystyle{siam}
\bibliography{references1}	

\end{document}